\documentclass[preprint,12pt]{elsarticle}
\usepackage{graphicx,floatrow}
\usepackage{amssymb,amsmath, amsthm, subfigure}
\usepackage{color}
\usepackage{mathrsfs}
\usepackage{amsmath}
\usepackage{hyperref}
\journal{}

\newtheorem{lemmaA}{Lemma}
\newtheorem{prop}{Proposition}

\begin{document}

\begin{frontmatter}

\title{A Generalized-Jacobi-Function Spectral Method for Space-Time Fractional Reaction-Diffusion Equations with Viscosity Terms}

\author {Zhe Yu, Boying Wu, Jiebao Sun$^{ *}$, Wenjie Liu}
\cortext[cor1]{Correspondence to: Jiebao Sun, Department of Mathematics, Harbin Institute of Technology, Harbin 150001, People's Republic of China (e-mail: sunjiebao@hit.edu.cn).}

\address{Department of Mathematics, Harbin Institute of Technology, Harbin 150001, P. R. China}

\newtheorem{corollary}{Corollary}[section]
\newtheorem{theorem}{Theorem}[section]
\newtheorem{proposition}{Proposition}[section]
\newtheorem{definition}{Definition}[section]
\newtheorem{lemma}{Lemma}[section]
\newtheorem{remark}{Remark}[section]
\numberwithin{equation}{section}
\setcounter{page}{1}
\catcode`@=12

%%%%%%%%%%%%%%%%%%%%%%%%%%%%abstract%%%%%%%%%%%%%%%%%%%%%%%%%

\begin{abstract}
In this work, we study a new spectral Petrov-Galerkin approximation of space-time fractional reaction-diffusion equations with viscosity terms built by Riemann-Liouville fractional-order derivatives. The proposed method is reliant on generalized Jacobi functions (GJFs) for our problems. The contributions are threefold: First, thanks to the theoretical framework of variational problems, the well-posedness of the problem is proved. Second, new GJF-basis functions are established to fit weak solutions, which take full advantages of the global properties of fractional derivatives. Moreover, the basis functions conclude singular terms, in order to solve our problems with given smooth source term. Finally, we get a numerical analysis of error estimates to depend on GJF-basis functions. Numerical experiments confirm the expected convergence. In addition, they are given to show the effect of the viscosity terms in anomalous diffusion.

\end{abstract}

%%%%%%%%%%%%%%%%%%%%%%%%%%%%keyword%%%%%%%%%%%%%%%%%%%%%%%%%%%%%%

\begin{keyword}
Viscosity term; Fractional derivative; Spectral method; Error estimate.
\end{keyword}

\end{frontmatter}

%%%%%%%%%%%%%%%%%%%%%%%%%%%1. Introduction%%%%%%%%%%%%%%%%%%%%%%%%%%

\section{Introduction}

Fractional partial differential equations (FPDEs) have been attached significance to their applications, including system biology, physics, chemistry and biochemistry, hydrology, finance etc \cite{IP}. Fractional-order derivatives and integrals enable the description of the memory and hereditary properties of different substances. Compare to the classical models, FPDEs lead to greater approval based on data gained in lab experiments \cite{IP}. Some famous FPDEs include the fractional Fokker-Planck equation \cite{BMK}, the time fractional Schr\"{o}dinger equation \cite{NM}, the fractional Ginzburg-Landau equation \cite{TZ}, the fractional quasi-geostrophic equation \cite{PG}, and the fractional Landau-Lifshitz equation \cite{GZ, PGZ} etc. As a kind of classical FPDEs, the space-time fractional diffusion equations have been successfully modeled in a great number of physical phenomena. The advanced applications contain the turbulent flow, chaotic dynamics charge transport in amorphous semiconductors, NMR diffusometry in disordered materials, and dynamics of a bead in polymer network \cite{LX2}. 

In recent decades viscosity terms have applied to many fields, such as mechanics \cite{LRF1, LRF2}, physics \cite{LFL}, chemistry engineering \cite{BR, RB} etc. Researchers built viscosity terms mainly by classical derivatives. More typically, Showalter and Ting \cite{ST} showed that $\Delta u_t$ could be used as a viscosity term. Recently, Mao and Shen \cite{MS} constructed a non-local viscosity term by a special order derivative for two water wave models, and the authors also gave a numerical study of their decays by a semi-implicit spectral deferred correction method.

Actually, only a few types of FPDEs could get an analytical solution. It seems to be more important and useful for investigating a both efficient and accurate numerical method for applications. A classical and popular method for solving one-direction fractional diffusion (time diffusion or space diffusion only) is to offer a finite difference, finite element or finite volume methods for discretizing the fractional part. Researchers aim to approximate fractional derivatives of the ideal points by using a small number of parameters instead of their global information. Zhang et al. \cite{ZSZ} proposed a compact alternating direction implicit scheme for solving fractional diffusion-wave equations in two dimensions. Gao and Sun \cite{GS} worked out the fractional sub-diffusion equations by a compact finite difference scheme. 
%The authors got a good consequence of convergence and stability. 
In order to reduce the cost of computation, Wang et al. \cite{WYZ} figured out the inhomogeneous Dirichlet boundary-value problems of space-fractional diffusion equations in finite element method. The space-fractional diffusion equations in fast finite volume methods and fast difference methods have been solved in \cite{WCW,JW}, the authors dropped require storage and computational cost from $O ( N^2 ) $ and $ O(N^3 )$ to $O ( N ) $ and $ O(N  \log  N )$. However, although researchers have got good results using local numerical methods, the approximations of the fractional parts largely waste their global properties. What is more, those local numerical methods are expensive to compute and invert because of their full and dense coefficient matrices. The costs of computation bring a power growth with the iteration increasing. 

As a typical global numerical method, the spectral method can offer greater global characteristics by way of constructing basis functions, rather than local numerical methods. More importantly, the spectral method has an exponential convergence with increasing dimension of the approximation. Shen et al. \cite{STW} came up with spectral methods depend on classical Jacobi polynomials and used in integer-order problems with Dirichlet, Neumann, and Robin (or mixed) boundary conditions. Guo et al. \cite{GSZ} generalized Jacobi and Laguerre polynomials, and proposed the generalized Gauss quadrature, which enlarges the research areas of singularity problems. Yu and Guo \cite{YG} worked out the fourth-order mixed inhomogeneous boundary value problems via spectral element method. The complexity of boundary conditions is extended. Wang et al. \cite{WZZ} suggested superconvergence points of Jacobi-Gauss-type spectral interpolation, which enhanced the convergence order in the neighborhood of those points. Sheng et al. \cite{SWG} led to a spectral method to solve nonlinear problems, the authors presented a multistep Legendre-Gauss spectral collocation method for nonlinear Volterra integral equations.

There is the truth that fractional definitions and spectral methods perform the same global feature. To settle the fractional problems, the spectral method using global basis functions seems more suitable for non-local problems. Li and Xu \cite{LX2, LX} proposed space-time spectral methods based on Jacobi polynomials and got a good convergence by classical Jacobi polynomials. Zhao and Zhang \cite{ZZ} got the superconvergence points of fractional spectral interpolation. Jiao et al. \cite{JWH} optimized spectral collocation methods, in order to reduce the computational cost of standard spectral collocation methods. Very recently, Chen et al. \cite{CSW,CSW2} put forward generalized Jacobi functions and Laguerre functions to fractional differential equations. At the same time, the authors kept the properties of their numerical solutions with the weak formulas. Mao and Shen \cite{MS2, MS3, MS4} utilized efficient spectral methods to solve some fractional problems, where there are pretty popular with researchers. 
%Bhrawy et al. \cite{EHB, HDB, BED} enlarged the research fields by shifted orthogonal polynomials. The authors showed great numerical methods in these fields. 

In particular, our method has the following prominent features:
\begin{itemize}
\item[(1)] We present a space-time global spectral method to approximate the FPDEs without discretization, which cut down the computational complexity.
\item[(2)] Our spectral basis functions consist of a combination of GJFs. Hence, the basis functions could fit the characteristics of trial space and test space.
\item[(3)] Thanks to the properties of GJFs, we could compute the Riemann-Liouville fractional derivatives of basis functions directly. At the same time, the straight calculation could reduce the computational cost relative to the finite approximations.
\item[(4)] By an accurate selection of trial space and test space, our space-time GJF-Petrov-Galerkin method performs a sparse enough linear system, which is cheap to solve our fractional problems.
\end{itemize}

In this paper, we adopt Riemann-Liouville fractional definitions due to the physics background of the equation. Our goal is mainly to solve a new space-time fractional reaction-diffusion equations with viscosity terms, by way of a new space-time GJF-Petrov-Galerkin spectral method. The rest of the paper is organized as follows. In Section 2, we shall put forward our problem and recall the definitions and several lemmas for constructing a weak formulation. The well-posedness of the problem is proved. In Section 3, we discuss in details about the space-time spectral method, including a numerical analysis of the error estimates. Numerical results and conclusions are shown in Section 4 and Section 5.

%%%%%%%%%%%%%%%%%%%%%%%%%%%%%%%%%2.Preliminaries%%%%%%%%%%%%%%%%%

\section{Problem and Weak Formulation}
Let $\Lambda  := ( { - 1,1} ),\ I := ( {0,T} )$ be standard space and time domains, and denote $Q_T:=\Lambda \times I$. Consider the following equation
\begin{equation}
\label{1-1} {_0}\partial _t^\alpha u( {x,t} ) - {_{-1}}\partial _x^{\beta}u( {x,t} ) - \varepsilon\cdot{_0}\partial _t^{\gamma} {_{-1}}\partial _x^{\mu }u( {x,t} ) + u( {x,t} ) = f( {x,t} ),\;( {x,t} ) \in {Q_T,}
\end{equation}
with initial value and Dirichlet boundary conditions
\begin{equation}
\label{1-2}u( {x,0} ) = 0,\; x \in \bar{\Lambda}, 
\end{equation}
\begin{equation}
\label{1-3}u( { - 1,t} ) = u( {1,t} ) = 0,\; t \in I,
\end{equation}
where $ \alpha, \gamma \in (0,1), \beta, \mu \in (1,2), \varepsilon \in [0,1] $.

\subsection{Preliminaries}

In this subsection, we concisely point out some definitions and lemmas about fractional Sobolev spaces with norms and inner products, and we chiefly discuss the existence and uniqueness based on these definitions and lemmas.

Let $c$ be a generally positive constant without any relationship with functions and discretization parameters. We use the expression $ A \lesssim (\gtrsim) B $ to represent $ A \leqslant  (\geqslant) cB $, 
% $ A \gtrsim B $ to mean that $ A \geqslant cB $, 
and $ A \cong B $ to stand for $ A \lesssim B \lesssim A $. Let $ \Theta $ be a normal domain like $ Q_T,\Lambda,I $ or $ \mathbb{R}$. $ {{L^{2}_{\omega}}} ( \Theta  ) $ space may be defined as a space of functions for which the 2nd power of the absolute value is Lebesgue integrable with weight function $\omega $. 
%$ \Theta $. 
The inner product and norm of $ {{L^{2}_{\omega}}} ( \Theta  ) $ can be defined by
$$ {{( u,v )}_{\omega,\Theta }} :=\int_{\Theta }{uv\omega\text{d}\Theta },\ \ {{\left\| u \right\|}_{\omega,\Theta }} :=( u,u )_{\omega,\Theta }^{\frac{1}{2}},\ \ \forall u,v\in {{L}^{2}_{\omega}}( \Theta  ) , \omega \in {{L}^{1}_{\text{loc}}}( \Theta  ). $$ 
%Similarly, let $ {H^m_{\omega}}( \Theta  )$ and $ H^m_{0,\omega}( \Theta  )\ $ be the ordinary Sobolev spaces with weight $\omega$, and their norms are denoted by $ {\left\| {\; \cdot \;} \right\|_{m,\omega,\Theta }}\ $.
Generally, we denote
$$ {{( u,v )}_{\Theta }}:={{( u,v )}_{\textbf{1},\Theta }}, $$
as weight function $\omega \equiv \textbf{1}$. We recall the definitions of some  Sobolev spaces (see \cite{IP,LX,EVR}). 
\begin{definition}[see \cite{LX,EVR}]
For real $ p \geqslant 0$, let
$$ {H^p} ( \mathbb{R} ) = \left\{ {\left. {\varphi( z )} \right|\varphi \in {L^2}( \mathbb{R} );\;{{( {1 + {{\left| \omega  \right|}^2}} )}^{\frac{p}{2}}}\mathscr{F}( \varphi )( \omega  ) \in {L^2}( \mathbb{R} )} \right\}, $$
endowed with the norm:
$$ {\left\| \varphi \right\|_{{H^p}( \mathbb{R}  )}} = {\left\| {{{( {1 + {{\left| \omega  \right|}^2}} )}^{\frac{p}{2}}}\mathscr{F}( \varphi )( \omega  )} \right\|_{{L^2}( \mathbb{R}  )}}, $$
where $ \mathscr{F}( \varphi ) $ denotes the Fourier transform of $ \varphi $. 
\end{definition}

For bounded domain $ \Omega = ({a,b}) $, which may present $ I $ or $\Lambda $, we get the following definitions.
\begin{definition}[see \cite{LX,EVR}]
For real $ p \geqslant 0$, we define the space:
$$ {H^p}( \Omega ) = \left\{ {\left. {\varphi \in {L^2}( \Omega )} \right|\exists \tilde \varphi \in {H^p}( \mathbb{R} ),\;{{\left. {\tilde \varphi} \right|}_{\Omega}} = \varphi} \right\},$$
with the norm:
$$ {\left\| \varphi \right\|_{{H^p}( \Omega  )}} = \mathop {\inf }\limits_{\tilde \varphi \in {H^p}( \mathbb{R} ),\;{{\left. {\tilde \varphi} \right|}_{\Omega}} = \varphi} {\left\| {\tilde \varphi} \right\|_{{H^p}( \mathbb{R}  )}}. $$
\end{definition}
\begin{definition}[see \cite{LX}]
Let $ C_0^\infty ( {\Omega} )$ be the space of smooth functions with compact support in $ {\Omega} $, and $ H_0^p( {\Omega} )$ represent the closure of $ C_0^\infty ( {\Omega} )$ with respect to norm $ {\left\| {\; \cdot \;} \right\|_{H^p ( {\Omega} ) }}$. Define
\begin{align*}
\begin{aligned}
_0{C^\infty }( \Omega ) & = \left\{ {\left. \varphi  \right|\varphi  \in {C^\infty }( \Omega )\ {\rm{with \  compact \  support \  in}}\ ( {a,b} ]} \right\},\\
^0{C^\infty }( \Omega ) & = \left\{ {\left. \varphi  \right|\varphi  \in {C^\infty }( \Omega )\ {\rm{with \  compact \ support \ in}}\ [ {a,b} )} \right\}.
\end{aligned}
\end{align*}
The space $ _0{H^p}( {\Omega} )$ (respectively, $ ^0{H^p}( {\Omega} )$) denotes the closure of $ _0{C^\infty }( {\Omega} )$ (resp. $^0{C^\infty }( \Omega )$) with respect to norm $ {\left\| {\; \cdot \;} \right\|_{H^p ( {\Omega} ) }}$. For the Sobolev space $ X $ with norm $ {\left\| {\; \cdot \;} \right\|_X}$, let
\begin{align*}
\begin{aligned}
{H^p}( {{\Omega};X} ) & := \left\{ {\varphi\left| {{{\left\| {\varphi( { \cdot ,t} )} \right\|}_X} \in {H^p}( {\Omega} )} \right.} \right\},\;p \geqslant 0,\\
_0{H^p}( {{\Omega};X} ) & := \left\{ {\varphi\left| {{{\left\| {\varphi( { \cdot ,t} )} \right\|}_X} \in {_0}{H^p}( {\Omega} )} \right.} \right\},\;p \geqslant 0,\\
^0{H^p}( {{\Omega};X} ) & := \left\{ {\varphi\left| {{{\left\| {\varphi( { \cdot ,t} )} \right\|}_X} \in {^0}{H^p}( {\Omega} )} \right.} \right\},\;p \geqslant 0,
\end{aligned}
\end{align*}
endowed with the norm:
$$ {\left\| \varphi \right\|_{{H^p}( {{\Omega};X} )}}: = {\left\| {{{\left\| {\varphi( { \cdot ,t} )} \right\|}_X}} \right\|_{H^p ( {\Omega} ) }}.$$
\end{definition}

We denote spaces:
\begin{align*}
{_0B^{s,\sigma ,\theta ,\rho }}( Q_T ) & := {{}_0}{H^s}( {I;{L^2}( \Lambda  )} ) \cap {L^2}( {I;H_0^\sigma ( \Lambda  )} ) \cap {H^\theta }( {I;H^\rho ( \Lambda  )} ), \\
{^0B^{s,\sigma ,\theta ,\rho }}( Q_T ) & := {{}^0}{H^s}( {I;{L^2}( \Lambda  )} ) \cap {L^2}( {I;H_0^\sigma ( \Lambda  )} ) \cap {H^\theta }( {I;H^\rho ( \Lambda  )} ),
\end{align*}
equipped with the norm:
\[{\left\| \varphi  \right\|_{{B^{s,\sigma ,\theta ,\rho }}( Q_T )}}: = {( {\left\| \varphi  \right\|_{{H^s}( {I;{L^2}( \Lambda  )} )}^2 + \left\| \varphi  \right\|_{{L^2}( {I;H_0^\sigma ( \Lambda  )} )}^2 + \left\| \varphi  \right\|_{{H^\theta }( {I;H^\rho ( \Lambda  )} )}^2} )^{\frac{1}{2}}}.\]
It can be verified that $ {_0B^{s,\sigma ,\theta ,\rho }}( Q_T  )$ and ${^0B^{s,\sigma ,\theta ,\rho }}( Q_T ) $ are both Banach spaces.

We will also recall some definitions of fractional derivatives and related properties.

\begin{definition}[Fractional integrals and derivatives, see \cite{IP}]%(see \cite{IP}) 
\label{def2-1}
For real $ p > 0 $, the left and right fractional integrals are respectively defined as
\begin{align}
\begin{aligned}
_aI_z^p \varphi ( z ) = \frac{1}{{\Gamma ( p  )}}\int_a^z {\frac{{\varphi ( \xi  )}}{{{{( {z - \xi } )}^{1 - p }}}}{\rm{d}}\xi} ,\;\;z > a,\\
_zI_b^p \varphi ( z ) = \frac{1}{{\Gamma ( p )}}\int_z^b {\frac{{\varphi ( \xi  )}}{{{{( {\xi  - z} )}^{1 - p }}}}{\rm{d}}\xi} ,\;\;z < b,
\end{aligned}
\end{align}
where $\Gamma( \, \cdot \, ) $ is the usual Gamma function.
 
For real number $ p \geqslant 0$, $ n - 1 \leqslant p < n$ with $ n \in \mathbb{N}$. The left and right Riemann-Liouville derivatives of order $p$ are, respectively, defined as:

left Riemann-Liouville derivative:
\begin{align}
\begin{aligned}
 _a\partial_z^p\varphi( z ) = \frac{1}{{\Gamma ( {n - p} )}}\frac{{{{\rm{d}}^n}}}{{{\rm{d}}{z^n}}}\int_a^z {\frac{{\varphi( \xi  )}}{{{{( {z - \xi } )}^{p - n + 1}}}}{\rm{d}}\xi } ,\;\forall z \in [ {a,b} ],\, 
\end{aligned}
\end{align}

right Riemann-Liouville derivative:
\begin{align}
\begin{aligned}
_z\partial_b^p\varphi( z ) = \frac{{( { - 1} )}^n}{\Gamma ( {n - p} )}\frac{{\rm{d}}^n}{{\rm{d}}{z^n}}\int_z^b {\frac{{\varphi( \xi  )}}{{{{( {\xi  - z} )}^{p - n + 1}}}}{\rm{d}}\xi } ,\;\forall z \in [ {a,b} ],\, 
\end{aligned}
\end{align}
where $ \frac{{\rm{d}}^n}{{\rm{d}}{z^n}} $ stands for the usual derivative of integer order $ n $.

For real number $ p \geqslant 0$, $ n - 1 \leqslant p < n$ with $ n \in \mathbb{N}$. The left and right Caputo derivatives of order $p$ are, respectively, defined as:

left Caputo derivative:
\begin{align}
\begin{aligned}
^C_a\partial_z^p\varphi( z ) = \frac{1}{{\Gamma ( {n - p} )}}\int_a^z \frac{1}{{{{( {z - \xi } )}^{p - n + 1}}}}\frac{{\rm{d}}^n\varphi}{{\rm{d}}{z^n}}{( \xi  )} {\rm{d}}\xi  ,\;\forall z \in [ {a,b} ],\, 
\end{aligned}
\end{align}

right Caputo derivative:
\begin{align}
\begin{aligned}
^C_z\partial_b^p\varphi( z ) = \frac{{{{( { - 1} )}^n}}}{{\Gamma ( {n - p} )}}\int_z^b {\frac{1}{{{{( {\xi  - z} )}^{p - n + 1}}}}\frac{{\rm{d}}^n\varphi}{{\rm{d}}{z^n}}{( \xi  )}{\rm{d}}\xi } ,\;\forall z \in [ {a,b} ]. 
\end{aligned}
\end{align}
These two definitions above are connected with the following relationship, which can be verified by partial integration:
\begin{align}
\label{transRiemCapu}
\begin{aligned}
_a\partial_z^p\varphi ( z ) = {}_a^C\partial_z^p\varphi ( z ) + \sum\limits_{j = 0}^{n - 1} {\frac{1}{{\Gamma ( {1 + j - p} )}}\frac{{\rm{d}}^j\varphi}{{\rm{d}}{z^j}}( a ){{( {z - a} )}^{j - p}}}, \\
_z\partial_b^p\varphi ( z ) = {}_z^C\partial_b^p\varphi ( z ) + \sum\limits_{j = 0}^{n - 1} \frac{{( { - 1} )}^j}{\Gamma ( {1 + j - p} )}\frac{{\rm{d}}^j\varphi}{{\rm{d}}{z^j}}( b ){( {z - b} )}^{j - p}.
\end{aligned}
\end{align}
\end{definition}

In some cases, we hope that the fractional derivative does not seem very mysterious, so we lead in some lemmas related to the Riemann-Liouville fractional derivatives. Thanks to these lemmas, they offer us a pretty convenient way to get corresponding results.

\begin{lemma}[see \cite{IP}]
\label{lemma2-1} Suppose that $ m-1 \leqslant p < m $, $ n-1 \leqslant q < n $, if ${\varphi}^{( j )} ( a ) = 0, \ j = 0,1,\cdots,r-1, r := \max\left\{m,n\right\}, z>a $, we have
\begin{align}
_a\partial{_z^p}{_a}\partial_z^q\varphi ( z ) = {}_a\partial{_z^q}{_a}\partial_z^p\varphi ( z ) = {_a}\partial_z^{p + q}\varphi ( z ).
\end{align}
\end{lemma}

\begin{lemma}[see \cite{IP}]
\label{lemma2-2} For $p \in [{k-1,k} ), k \in \mathbb{N} $, we have
\begin{align}
\begin{aligned}
( {_a\partial_z^pu,v} ) & = ( {u,{}_z^C\partial_b^pv} ) + \sum\limits_{j = 0}^{k - 1} {\left. {{{( { - 1} )}^j}\frac{{\rm{d}}^{j}v}{{\rm{d}}{z^{j}}}( z )\frac{{\rm{d}}^{k-j-1}}{{\rm{d}}{z^{k-j-1}}}{{}_a}I_z^{k - p}u( z )} \right|_{z = a}^{z = b}}, \\
( {_z\partial_b^pu,v} ) & = ( {u,{}_a^C\partial_z^pv} ) + \sum\limits_{j = 0}^{k - 1} {\left. {{{( { - 1} )}^{k - j}}\frac{{\rm{d}}^{j}v}{{\rm{d}}{z^{j}}}( z )\frac{{\rm{d}}^{k-j-1}}{{\rm{d}}{z^{k-j-1}}}{{}_z}I_b^{k - p}u( z )} \right|_{z = a}^{z = b}}.
\end{aligned}
\end{align}
\end{lemma}

\begin{lemma} %(see \cite{LX})
\label{lemma2-3} For real $p \in ({0,2} ), p \neq 1, u \in {_0}{H^\frac{p}{2}}( {\Omega } ),\;v \in {^0}{H^{\frac{p}{2}}}( { {\Omega } } )$, we have
\begin{align}
{\left\langle {_a\partial_z^pu,v} \right\rangle _\Omega } = {( {_a\partial_z^{\frac{p}{2}}u,{{}_z}\partial_b^{\frac{p}{2}}v} )_\Omega }.
\end{align}
\end{lemma}
\begin{proof}
Suppose $w = {_a}\partial_x^{\frac{p}{2}}u$, we can get $_a\partial_z^pu = {{}_a}\partial_x^{\frac{p}{2}}w$ by Lemma \ref{lemma2-1}. Thanks to the Lemma \ref{lemma2-2} and \eqref{transRiemCapu}, it shows that
\begin{align}
\begin{aligned}
\left\langle {_a\partial_z^pu,v} \right\rangle_\Omega & = \left\langle {_a\partial_z^{\frac{p}{2}}w,v} \right\rangle_\Omega = ( {w,{}_z^C\partial_b^{\frac{p}{2}}v} )_\Omega + \left. {v( z ) {_a}I_z^{1 - \frac{p}{2}}w( z )} \right|_{z = a}^{z = b}\\
& = ( {w,{_z}\partial_b^{\frac{p}{2}}v} )_\Omega - {\left. {v( a ) {{}_a}I^{1 - \frac{p}{2}}_zw( z )} \right|_{z = a}}.
\end{aligned}
\end{align}
By Definition \ref{def2-1}, we can easily get ${\left. {{_a}I_z^{1 - \frac{p}{2}}w( z )} \right|_{z = a}} = 0$, which ends the proof.
\end{proof}

For the sake of convenience, we propose the following definitions.
\begin{definition}
\label{defllff}
Let $p,q >0$, we define the seminorm:
\[{\left| \varphi  \right|_{H_l^p( {I;H_l^q ( \Lambda  )} )}}: = {\left\| {_0\partial _t^p{_{ - 1}}\partial _x^q \varphi } \right\|_{{L^2}( {I;{L^2}( \Lambda  )} )}},\]
and norm
\[{\left\| \varphi  \right\|_{H_l^p( {I;H_l^q ( \Lambda  )} )}}: = {( {\left\| \varphi  \right\|_{{L^2}( {I;{L^2}( \Lambda  )} )}^2 + \left| \varphi  \right|_{H_l^p( {I;H_l^q ( \Lambda  )} )}^2} )^{\frac{1}{2}}}.\]
Then we define $H_l^p( {I;H_l^q ( \Lambda  )} )$ as the closure of $C_0^\infty ( {I;C_0^\infty ( \Lambda  )} )$ with respect to norm ${\left\| {\; \cdot \;} \right\|_{H_l^p( {I;H_l^q ( \Lambda  )} )}}$. Similarly, we can also define spaces $ H_l^p( {I;H_r^q ( \Lambda  )} ),H_r^p( {I;H_l^q ( \Lambda  )} ),H_r^p( {I;H_r^q ( \Lambda  )} ) $ with their seminorms and norms.
\end{definition}

\begin{definition}
\label{deflcff}
Let $p,q >0, q \neq n + \frac{1}{2}$, we define the seminorm:
\[{\left| \varphi  \right|_{H_l^p( {I;H_c^q ( \Lambda  )} )}}: = {\left| {{{( {_0\partial _t^p{_{ - 1}}\partial _x^q \varphi ,{_0}\partial _t^p{_x}\partial _1^q \varphi } )}_{{L^2}( Q_T )}}} \right|^{\frac{1}{2}}},\]
and norm
\[{\left\| \varphi  \right\|_{H_l^p( {I;H_c^q ( \Lambda  )} )}}: = {( {\left\| \varphi  \right\|_{{L^2}( {I;{L^2}( \Lambda  )} )}^2 + \left| \varphi  \right|_{H_l^p( {I;H_c^q( \Lambda  )} )}^2} )^{\frac{1}{2}}}.\]
Then we define $H_l^p( {I;H_c^q ( \Lambda  )} )$ as the closure of $C_0^\infty ( {I;C_0^\infty ( \Lambda  )} )$ with respect to norm ${\left\| {\; \cdot \;} \right\|_{H_l^p( {I;H_c^q ( \Lambda  )} )}}$. Similarly, we can also define spaces $ H_r^p( {I;H_c^q ( \Lambda  )} ), H_c^p( {I;H_l^q ( \Lambda  )} ), H_c^p( {I;H_r^q( \Lambda  )} )$ with their seminorms and norms.
\end{definition}

\begin{definition}
\label{defccff}
Let $p,q>0, p,q \ne n + \frac{1}{2}$, we define the seminorm:
\[{\left| \varphi  \right|_{H_c^p( {I;H_c^q ( \Lambda  )} )}}: = {\left| {{{( {_0\partial _t^p{_{ - 1}}\partial _x^q \varphi ,{_t}\partial _T^p{_x}\partial _1^q \varphi } )}_{{L^2}( Q_T )}}} \right|^{\frac{1}{2}}},\]
and norm
\[{\left\| \varphi  \right\|_{H_c^p( {I;H_c^q ( \Lambda  )} )}}: = {( {\left\| \varphi  \right\|_{{L^2}( {I;{L^2}( \Lambda  )} )}^2 + \left| \varphi  \right|_{H_c^p( {I;H_c^q( \Lambda  )} )}^2} )^{\frac{1}{2}}}.\]
Then we define $H_c^p( {I;H_c^q ( \Lambda  )} )$ as the closure of $C_0^\infty ( {I;C_0^\infty ( \Lambda  )} )$ with respect to norm ${\left\| {\; \cdot \;} \right\|_{H_c^p( {I;H_c^q ( \Lambda  )} )}}$. 
\end{definition}

%%%%%%%%%%%%%%%%%3.%Weak solution and its existence and uniqueness%%%%%%%%%%
\subsection{Existence and Uniqueness of Weak Solutions}
In this subsection we shall discuss the weak formulation of \eqref{1-1}. We will establish the formulation based on Lemmas \ref{lemma2-2}-\ref{lemma2-3}.

\begin{definition}
\label{def3-1} A function $ u\in {_0B^{\frac{\alpha }{2},\frac{\beta }{2},\frac{\gamma }{2},\frac{\mu }{2}}}( Q_T )$ is said to be a weak solution of equation \eqref{1-1}, if for any $ v\in {^0B^{\frac{\alpha }{2},\frac{\beta }{2},\frac{\gamma }{2},\frac{\mu }{2}}}( Q_T ) $, the following equation
\begin{equation}
\label{3-1}  \mathscr{A}_\varepsilon ^{(\alpha ,\beta ,\gamma ,\mu ) }( {u,v} )=F( v ),\ \forall v\in {^0B^{ \frac{\alpha }{2},\frac{\beta }{2},\frac{\gamma }{2},\frac{\mu }{2}}}( Q_T )
\end{equation}
is valid where $\mathscr{A}_\varepsilon ^{(\alpha ,\beta ,\gamma ,\mu ) }( u,v ) $ is a bilinear form on ${_0B^{\frac{\alpha }{2},\frac{\beta }{2},\frac{\gamma }{2},\frac{\mu }{2}}}( Q_T )\times {{}^0B^{\frac{\alpha }{2},\frac{\beta }{2},\frac{\gamma }{2},\frac{\mu }{2}}}( Q_T ) $, satisfies
\[\mathscr{A}_\varepsilon ^{(\alpha ,\beta ,\gamma ,\mu ) }( {u,v} ) = {( {_0\partial _t^{\frac{\alpha }{2}}u,{_t}\partial _T^{\frac{\alpha }{2}}v} )_{Q_T}} - {( {_{ - 1}\partial _x^{\frac{\beta }{2}}u,{_x}\partial _1^{\frac{\beta }{2}}v} )_{Q_T}} - \varepsilon  \cdot {( {_0\partial _t^{\frac{\gamma }{2}}{_{ - 1}}\partial _x^{\frac{\mu }{2}}u,{_t}\partial _T^{\frac{\gamma }{2}}{_x}\partial _1^{\frac{\mu }{2}}v} )_{Q_T}} + {( {u,v} ) }_{Q_T},\]
and the functional $ F( \, \cdot \,  )$ is given by 
\[F( v ):= {\left\langle {f,v} \right\rangle _{Q_T}}.\]
\end{definition}

Hereafter, we denote
\[{s} = \max \left\{ {\frac{\alpha}{2} ,\frac{\gamma}{2} } \right\},\;\;\;\;{\sigma} = \max \left\{ {\frac{\beta}{2} ,\frac{\mu}{2} } \right\}.\]
Then we give the following theorem.
\begin{theorem}
\label{thm3-1} For all $ \alpha ,\gamma  \in ( {0,1} ),\;\beta ,\mu  \in ( {1,2} ),\;\varepsilon  \in [ {0,\min\{2^{\gamma - \alpha},1 \} } ] $, $ f \in {( {{^0B^{\frac{\alpha }{2},\frac{\beta }{2},\frac{\gamma }{2},\frac{\mu }{2}}}( Q_T )} )^\prime } $, 
%and $u$ satisfies
%\begin{align}
%\label{weakformulationcondition}
%\int_{-1}^1 u(x,t) \omega^{(-\sigma,0)} {\rm d}x = 0, \ \ \ \ \ \ \ \ \ \ \int_0^T u(x,t) \bar{\omega}^{(-s,0)} {\rm d}t = 0, 
%\end{align}
then the problem (2.14) is well-posed. Furthermore, if $u$ is the solution of problem (2.14), then it holds
\begin{equation}
\label{3-2}
\left\| u \right\|_{{B^{\frac{\alpha }{2},\frac{\beta }{2},\frac{\gamma }{2},\frac{\mu }{2}}}( Q_T )} \lesssim {\left\| f \right\|_{{{\left( {{B^{\frac{\alpha }{2},\frac{\beta }{2},\frac{\gamma }{2},\frac{\mu }{2}}}( Q_T )} \right)}^\prime }}}.
\end{equation}
\end{theorem}
The details of the proof would be presented in \ref{AppendixA}.

%%%%%%%%%%%%%%%%%%%%%%%%%%%%%%%%%4.Spectral Garlerkin method%%%%%%%%%%%%%%%%%
\section{Spectral GJF-Petrov-Galerkin method}

In this section, we aim to solve the proposed problem, which conveyed by its weak solution \eqref{3-1}, with initial value and Dirichlet boundary conditions via a space-time spectral Petrov-Galerkin method. We bring in some approximation operators and derive the corresponding approximation results, and compute the error estimation of our numerical solution.

\subsection{Main Algorithm}
In this subsection we shall mainly solve the problem in the weak formulation \eqref{3-1} with the space-time spectral Petrov-Galerkin method. In order to solve the weak solution \eqref{3-1} without singularity by fractional derivatives, we here construct spectral basis functions depend on GJFs. At the beginning, denote $\omega^{(\tilde{\alpha},\tilde{\beta})} ( x ) = {(1-x)}^{\tilde{\alpha}}{(1+x)}^{\tilde{\beta}} $ and $\bar{\omega}^{( \tilde{\alpha}, \tilde{\beta})} ( x ) = {t}^{\tilde{\beta}}{(T-t)}^{\tilde{\alpha}} $. We introduce the definitions and properties of GJFs.

\begin{definition}[\cite{CSW}]
Let $P_n^{( {{\tilde \alpha } ,{\tilde \beta } } )}( x )$ be the generalized Jacobi polynomials with parameters ${\tilde \alpha }, {\tilde \beta }$. Define
\begin{align}
^ + J_n^{( { - {\tilde \alpha } ,{\tilde \beta } } )}( x ): = {( {1 - x} )^{\tilde \alpha } }P_n^{( {{\tilde \alpha } ,{\tilde \beta } } )}( x ),\;\;\;\;{\tilde \alpha }  >  - 1,\;{\tilde \beta }  \in \mathbb{R},\\
^ - J_n^{( {{\tilde \alpha } , - {\tilde \beta } } )}( x ): = {( {1 + x} )^{\tilde \beta } }P_n^{( {{\tilde \alpha } ,{\tilde \beta } } )}( x ),\;\;\;\;{\tilde \alpha }  \in \mathbb{R},\;{\tilde \beta }  >  - 1,
\end{align}
as generalized Jacobi functions, where $x \in \Lambda $, and $n$ presents non-negative integers.
\end{definition}
\begin{remark}
A direct calculation leads to classical cases $(\tilde \alpha,\tilde \beta  > -1)$ that
\begin{align}
& ^ + J_n^{( { - \tilde \alpha ,\tilde \beta } )}( { - 1} ) = {2^{\tilde \alpha }}{( { - 1} )^n}\frac{{\Gamma ( {n + \tilde \beta  + 1} )}}{{n!\Gamma ( {\tilde \beta  + 1} )}},\;\;{{}^ + }J_n^{( { - \tilde \alpha ,\tilde \beta } )}( 1 ) = 0,\\
& ^ - J_n^{( {\tilde \alpha , - \tilde \beta } )}( { - 1} ) = 0,\;\;\;\;\;\;\;\;\;\;\;{{}^ - }J_n^{( {\tilde \alpha , - \tilde \beta } )}( 1 ) = {2^{\tilde \beta }}\frac{{\Gamma ( {n + \tilde \alpha  + 1} )}}{{n!\Gamma ( {\tilde \alpha  + 1} )}}.
\end{align}
\end{remark}

For calculating the derivatives of GJFs conveniently, we lead in the following lemma.
\begin{lemma}[\cite{CSW}]
\label{FDJF}
Let $\tilde s \in \mathbb{R}^+, x \in \Lambda $ and $n$ presents non-negative integers. 
\begin{itemize}
\item For ${\tilde \alpha } > \tilde s-1 $ and ${\tilde \beta } \in \mathbb{R} $,
\begin{align}
_x\partial_1^{\tilde s}\left\{ {^ + J_n^{( { - {\tilde \alpha } ,{\tilde \beta } } )}( x )} \right\} = \frac{{\Gamma ( {n + {\tilde \alpha }  + 1} )}}{{\Gamma ( {n + {\tilde \alpha }  - \tilde s + 1} )}}{{}^ + }J_n^{( { - {\tilde \alpha }  + \tilde s,{\tilde \beta }  + \tilde s} )}( x ).
\end{align}
\item For ${\tilde \alpha } \in \mathbb{R} $ and ${\tilde \beta } > \tilde s-1 $,
\begin{align}
_{ - 1}\partial_x^{\tilde s}\left\{ {^ - J_n^{( {{\tilde \alpha } , - {\tilde \beta } } )}( x )} \right\} = \frac{{\Gamma ( {n + {\tilde \beta } + 1} )}}{{\Gamma ( {n + {\tilde \beta }  - \tilde s + 1} )}}{{}^ - }J_n^{( {{\tilde \alpha }  + \tilde s, - {\tilde \beta }  + \tilde s} )}( x ).
\end{align}
\end{itemize}
\end{lemma}

Then, we define the spectral basis functions of the trial functions
\begin{align}
{\varphi _i}( x ) & = \frac{{\Gamma ( i )}}{{\Gamma ( {i + \sigma } )}}\left( {^ - J_{i - 1}^{( { - \sigma , - \sigma } )}( x ) - \frac{i}{{i - \sigma }}{\;^ - }J_i^{( { - \sigma , - \sigma } )}( x )} \right),\\
{\psi _j}( t ) & = \frac{{\Gamma ( j )}}{{\Gamma ( {j + s} )}}{\left( {\frac{T}{2}} \right)^{s - \frac{1}{2}}}{\left( {j - \frac{1}{2}} \right)^{\frac{1}{2}}}{\;^ - }J_{j - 1}^{( { - s, - s} )}\left( {\frac{{2t}}{T} - 1} \right),
\end{align}
and spectral basis functions of test functions
\begin{align}
{{\bar \varphi }_{i'}}( x ) & = \frac{{\Gamma ( {i'} )}}{{\Gamma ( {{i'} + \sigma } )}}\left( {^ + J_{{i'} - 1}^{( { - \sigma , - \sigma } )}( x ) + \frac{{i'}}{{{i'} - \sigma }}{{}^ + }J_{i'}^{( { - \sigma , - \sigma } )}( x )} \right),\\
{{\bar \psi }_{j'}}( t ) & = \frac{{\Gamma ( {j'} )}}{{\Gamma ( {{j'} + s} )}}{\left( {\frac{T}{2}} \right)^{s - \frac{1}{2}}}{\left( {{j'} - \frac{1}{2}} \right)^{\frac{1}{2}}}{{}^ + }J_{{j'} - 1}^{( { - s, - s} )}\left( {\frac{{2t}}{T} - 1} \right).
\end{align}
Furthermore, we define trial and test function spectral spaces.
\begin{align*}
{X_M}( \Lambda  ): = {\rm{span}} \left\lbrace {{\varphi _i}( x ):1 \leqslant i \leqslant M - 1,\;x \in \Lambda } \right\rbrace , & \;{Y_N}( I ): = {\rm{span}} \left\lbrace {{\psi _j}\left( {\frac{{2t}}{T} - 1} \right):1 \leqslant j \leqslant N,t \in I} \right\rbrace ;\\
{{\bar X}_M}( \Lambda  ): = {\rm{span}} \left\lbrace {{{\bar \varphi }_{i'}}( x ):1 \leqslant {i'} \leqslant M - 1,\;x \in \Lambda } \right\rbrace , & \;{{\bar Y}_N}( I ): = {\rm{span}} \left\lbrace {{{\bar \psi }_{j'}}\left( {\frac{{2t}}{T} - 1} \right):1 \leqslant j' \leqslant N,t \in I} \right\rbrace .
\end{align*}
Obviously, let
\[{V_L}( Q_T ): = {X_M}( \Lambda  ) \otimes {Y_N}( I ),\;\;{{\bar V}_L}( Q_T ): = {{\bar X}_M}( \Lambda  ) \otimes {{\bar Y}_N}( I )\]
be the solution space with the pair of positive integers $L:=({M,N})$. These basis functions satisfy the bounded conditions, i.e.
\begin{align*}
{\varphi _i}( { - 1} ) & = \frac{{\Gamma ( i )}}{{\Gamma ( {i + \sigma } )}}( {^ - J_{i - 1}^{( { - \sigma , - \sigma } )}( { - 1} ) - \frac{i}{{i - \sigma }}{{}^ - }J_i^{( { - \sigma , - \sigma } )}( { - 1} )} ) = 0,\\
{\varphi _i}( 1 ) &  = \frac{{\Gamma ( i )}}{{\Gamma ( {i + \sigma } )}}( {^ - J_{i - 1}^{( { - \sigma , - \sigma } )}( 1 ) - \frac{i}{{i - \sigma }}{{}^ - }J_i^{( { - \sigma , - \sigma } )}( 1 )} )\\
 &  = \frac{{\Gamma ( i )}}{{\Gamma ( {i + \sigma } )}}{2^\sigma }\left( {\frac{{\Gamma ( {i - \sigma } )}}{{\Gamma ( i )\Gamma ( {1 - \sigma } )}} - \frac{i}{{i - \sigma }}\frac{{\Gamma ( {i + 1 - \sigma } )}}{{\Gamma ( {i + 1} )\Gamma ( {1 - \sigma } )}}} \right) = 0,\\
{\psi _j}( 0 )  & = \frac{{\Gamma ( j )}}{{\Gamma ( {j + s} )}}{\left( {\frac{T}{2}} \right)^{s - \frac{1}{2}}}{\left( {j - \frac{1}{2}} \right)^{\frac{1}{2}}}{{}^ - }J_{j - 1}^{( { - s, - s} )}( { - 1} ) = 0; \\
{{\bar \varphi }_{i'}}( { - 1} ) &  = \frac{{\Gamma ( {i'} )}}{{\Gamma ( {{i'} + \sigma } )}}( {^ + J_{{i'} - 1}^{( { - \sigma , - \sigma } )}( { - 1} ) + \frac{{i'}}{{{i'} - \sigma }}{{}^ + }J_{i'}^{( { - \sigma , - \sigma } )}( { - 1} )} )\\
 &  = \frac{{\Gamma ( {i'} )}}{{\Gamma ( {{i'} + \sigma } )}}{2^\sigma }{( { - 1} )^{i'}}\left( {\frac{{\Gamma ( {{i'} - \sigma } )}}{{\Gamma ( {i'} )\Gamma ( {1 - \sigma } )}} - \frac{{i'}}{{{i'} - \sigma }}\frac{{\Gamma ( {{i'} + 1 - \sigma } )}}{{\Gamma ( {{i'} + 1} )\Gamma ( {1 - \sigma } )}}} \right) = 0,\\
{{\bar \varphi }_{i'}}( 1 ) &  = \frac{{\Gamma ( {i'} )}}{{\Gamma ( {{i'} + \sigma } )}}( {^ + J_{{i'} - 1}^{( { - \sigma , - \sigma } )}( 1 ) + \frac{{i'}}{{{i'} - \sigma }}{{}^ + }J_{i'}^{( { - \sigma , - \sigma } )}( 1 )} ) = 0,\\
{{\bar \psi }_{j'}}( T ) &  = \frac{{\Gamma ( {j'} )}}{{\Gamma ( {{j'} + s} )}}{\left( {\frac{T}{2}} \right)^{s - \frac{1}{2}}}{\left( {{j'} - \frac{1}{2}} \right)^{\frac{1}{2}}}{{}^ + }J_{{j'} - 1}^{( { - s, - s} )}( 1 ) = 0.
\end{align*}
Thanks to Lemma \ref{FDJF}, it is convenient to calculate that
\begin{align}
\label{basictest-dxdt}
_{ - 1}\partial_x^\sigma {\varphi _i}( x ) & = {L_{i - 1}}( x ) - \frac{{i + \sigma }}{{i - \sigma }}{L_i}( x ),\;{{}_0}\partial_t^s{\psi _j}( t ) = {\left( {\frac{T}{2}} \right)^{ - \frac{1}{2}}}{\left( {j - \frac{1}{2}} \right)^{\frac{1}{2}}}{L_{j - 1}}\left( {\frac{{2t}}{T} - 1} \right),\\
_x\partial_1^\sigma {{\bar \varphi }_{i'}}( x ) & = {L_{{i'} - 1}}( x ) + \frac{{{i'} + \sigma }}{{{i'} - \sigma }}{L_{i'}}( x ),\;{{}_t}\partial_T^s{{\bar \psi }_{j'}}( t ) = {\left( {\frac{T}{2}} \right)^{ - \frac{1}{2}}}{\left( {{j'} - \frac{1}{2}} \right)^{\frac{1}{2}}}{L_{{j'} - 1}}\left( {\frac{{2t}}{T} - 1} \right),
\end{align}
where $\left\lbrace L_n (x) \right\rbrace_{n=0}^{+\infty}$ present Legendre polynomials. It shows that
\begin{align*}
\begin{aligned}
& {( {_{ - 1}\partial_x^\sigma {\varphi _i},{{}_x}\partial_1^\sigma {{\bar \varphi }_{i'}}} )_{\Lambda }} = \left\{ {\begin{aligned}
{\frac{2}{{2i - 1}} - {{\left( {\frac{{i + \sigma }}{{i - \sigma }}} \right)}^2}\frac{2}{{2i + 1}},\;\;i = i',}\\
{ - \frac{{i + \sigma }}{{i - \sigma }} \cdot \frac{2}{{2i + 1}},\;\;\;\;\;\;\;\;\;\;\;i = i' - 1,}\\
{\frac{{i' + \sigma }}{{i' - \sigma }} \cdot \frac{2}{{2i' + 1}},\;\;\;\;\;\;\;\;\;\;\;\;i = i' + 1,}\\
{0,\;\;\;\;\;\;\;\;\;\;\;\;\;\;\;\;\;\;\;\;\;\;\;\;\;\;\;\;\;\;\;\;{\rm{otherwise},}}
\end{aligned}} \right.\\
& {( {_0\partial_t^s{\psi _j},{{}_t}\partial_T^s{{\bar \psi }_n}} )_{I}} = {\delta _{jj'}}.
\end{aligned}
\end{align*}
More generally,
\begin{align*}
_{ - 1}\partial_x^\rho {\varphi _i}( x ) & = \frac{{\Gamma ( i )}}{{\Gamma ( {i + \sigma  - \rho } )}}\left( {^ - J_{i - 1}^{( { - \sigma  + \rho , - \sigma  + \rho } )}( x ) - \frac{{i( {i + \sigma } )}}{{( {i - \sigma } )( {i + \sigma  - \rho } )}}{{}^ - }J_i^{( { - \sigma  + \rho , - \sigma  + \rho } )}( x )} \right),\\
_x\partial_1^\rho {{\bar \varphi }_{i'}}( x )  & = \frac{{\Gamma ( {i'} )}}{{\Gamma ( {i' + \sigma  - \rho } )}}\left( {^ + J_{i' - 1}^{( { - \sigma  + \rho , - \sigma  + \rho } )}( x ) + \frac{{i'( {i' + \sigma } )}}{{( {i' - \sigma } )( {i' + \sigma  - \rho } )}}{{}^ + }J_{i'}^{( { - \sigma  + \rho , - \sigma  + \rho } )}( x )} \right),\\
_0\partial_t^r{\psi _j}( t )  & = \frac{{\Gamma ( j )}}{{\Gamma ( {j + s - r} )}}{\left( {\frac{T}{2}} \right)^{s - r - \frac{1}{2}}}{\left( {j - \frac{1}{2}} \right)^{\frac{1}{2}}}{{}^ - }J_{j - 1}^{( { - s + r, - s + r} )}\left( {\frac{{2t}}{T} - 1} \right),\\
_t\partial_T^r{{\bar \psi }_{j'}}( t )  & = \frac{{\Gamma ( {j'} )}}{{\Gamma ( {j' + s - r} )}}{\left( {\frac{T}{2}} \right)^{s - r - \frac{1}{2}}}{\left( {j' - \frac{1}{2}} \right)^{\frac{1}{2}}}{{}^ + }J_{j' - 1}^{( { - s + r, - s + r} )}\left( {\frac{{2t}}{T} - 1} \right).
\end{align*}

Our goal here is to find ${{u}_{L}}\in {V_{L}( Q_T )} $ in problem \eqref{3-1} by a GJF-space-time spectral Petrov-Galerkin approximation, such that
\begin{align}
\label{4-5}
\mathscr{A}_\varepsilon ^{(\alpha ,\beta ,\gamma ,\mu )}( u_L,v_L )=F( {{v}_{L}} ),\ \forall v_L\in {\bar{V}}_L( Q_T ). 
\end{align}
Since ${V_{L}( Q_T )}$ (${\bar{V}_L( Q_T )}$ \textit{resp.}) is a subspace of $ {{}_0B^{\frac{\alpha }{2},\frac{\beta }{2},\frac{\gamma }{2},\frac{\mu }{2}}}( Q_T )$ ($ {{}^0B^{\frac{\alpha }{2},\frac{\beta }{2},\frac{\gamma }{2},\frac{\mu }{2}}}( Q_T )$ \textit{resp.}), the well-posedness of the Galerkin formulation \eqref{4-5} can be established similarly as in the continuous case \eqref{3-1}.

\begin{theorem}
\label{thm4-1} For all $ \alpha,\gamma  \in ( {0,1} ),\;\beta ,\mu  \in ( {1,2} ),\;\varepsilon  \in [ {0,\min\{2^{\gamma - \alpha},1 \} } ] $, $ f \in {( {{B^{\frac{\alpha }{2},\frac{\beta }{2},\frac{\gamma }{2},\frac{\mu }{2}}}( Q_T )} )^\prime } $, 
%and $u_L$ satisfies
%\begin{align*}
%\label{numericalweakformulationcondition}
%\int_{-1}^1 u_L(x,t) \omega^{(-\sigma,0)} {\rm d}x = 0, \ \ \ \ \ \ \ \ \ \ \int_0^T u_L(x,t) \bar{\omega}^{(-s,0)} {\rm d}t = 0, 
%\end{align*}
the problem \eqref{4-5} is well-posed. Furthermore, if $u_{L}$ is the solution of problem \eqref{4-5}, then it holds the stability that
\begin{equation}
\label{4-6}
\left\| {u}_{L} \right\|_{{B^{\frac{\alpha }{2},\frac{\beta }{2},\frac{\gamma }{2},\frac{\mu }{2}}}( Q_T )} \lesssim {\left\| f \right\|_{{{\left( {{B^{\frac{\alpha }{2},\frac{\beta }{2},\frac{\gamma }{2},\frac{\mu }{2}}}( Q_T )} \right)}^\prime }}}.
\end{equation}
\end{theorem}

In order to calculate integrals suitably and enough accurately, we should lead in suitable Gauss quadrature and ensure that the precision of numerical solution would not influence that of this algorithm. For this reason, we will use Jacobi-Gauss-type quadrature \textit{(see Theorem 3.25-3.27 in \cite{STW})} to complete the calculation.
At present, we construct the weak formulation, i.e., find ${u}_{L}\in {V}_{L}$, and for all $ {v}_{L}\in {\bar{V}}_{L}$, such that
\begin{align}
\label{WF}
{( {_0\partial _t^{\frac{\alpha }{2}}{u_L},{_t}\partial _T^{\frac{\alpha }{2}}{v_L}} )_{Q_T}} - {( {_{ - 1}\partial _x^{\frac{\beta }{2}}{u_L},{_x}\partial _1^{\frac{\beta }{2}}{v_L}} )_{Q_T}} - \varepsilon  \cdot {( {_0\partial _t^{\frac{\gamma }{2}}{_{ - 1}}\partial _x^{\frac{\mu }{2}}{u_L},{_t}\partial _T^{\frac{\gamma }{2}}{_x}\partial _1^{\frac{\mu }{2}}{v_L}} )_{Q_T}} + {( {u_L},{v_L} )_{Q_T}}= {( {f,{v_L}} )_{Q_T}}.
\end{align}

First, for separating variables, we need to choose two classes of suitable basis functions in different directions. We construct the exact solution $ {u}( x,t ) $ with the following format
\begin{equation}
\label{ES}
{u}( x,t )=\sum\limits_{i=1}^{+\infty}{\sum\limits_{j=1}^{+\infty}{{{u}_{ij}}{{\varphi }_{i}}( x ){{\psi }_{j}}( t )}},
\end{equation}
and numerical solution $ {u}_{L}( x,t ) $ respectively
\begin{equation}
\label{NS}
{u}_{L}( x,t )=\sum\limits_{i=1}^{M-1}{\sum\limits_{j=1}^{N}{{{\check{u}}_{ij}}{{\varphi }_{i}}( x ){{\psi }_{j}}( t )}}.
\end{equation}
Combine the format \eqref{NS} with the weak formulation \eqref{WF}, let ${v_L}( {x,t} ) = {{\bar \varphi }_{i'}}( x ){{\bar \psi }_{j'}}( t )$ for all $i' = 1,2,\cdots,M-1; \; j' = 1,2,\cdots,N$, we could write as a matrix form:
\begin{align}
P_0^T\check{U}{Q_{\frac{\alpha }{2}}} - P_{\frac{\beta }{2}}^T\check{U}{Q_0} - \varepsilon  \cdot P_{\frac{\mu }{2}}^T\check{U}{Q_{\frac{\gamma }{2}}} + P_0^T\check{U}{Q_0}= F,
\end{align}
where
\[\check{U} = {( {{\check{u}_{ij}}} )_{( {M - 1} ) \times N}}\]
is willing to be solved, and
\[{P_\rho } = {( {p_{ii'}^{( \rho  )}} )_{( {M - 1} ) \times ( {M - 1} )}},\;\;{Q_r} = {( {q_{jj'}^{( r )}} )_{N \times N}},\;\;F = {( {{f_{i'j'}}} )_{( {M - 1} ) \times N}}\]
satisfy
\[p_{ii'}^{( \rho  )} = {( {_{ - 1}\partial_x^\rho {\varphi _i},{{}_x}\partial_1^\rho {{\bar \varphi }_{i'}}} )_\Lambda },\;\;q_{jj'}^{( r )} = {( {_0\partial_t^r{\psi _j},{{}_t}\partial_T^r{{\bar \psi }_{j'}}} )_I},\;\;{f_{i'j'}} = {( {f,{{\bar \varphi }_{i'}}{{\bar \psi }_{j'}}} )_{Q_T}}.\]

Similarly, take
\[{f_{i'j'}} = \sum\limits_{\bar m = 1}^{M - 1} {\sum\limits_{\bar n = 1}^N {f( {{x_{\bar m}},{t_{\bar n}}} )\frac{{{{\bar \varphi }_{i'}}( {{x_{\bar m}}} )}}{{{{( {1 - {x_{\bar m}}} )}^\sigma }}}\frac{{{{\bar \psi }_{j'}}( {{t_{\bar n}}} )}}{{{{( {T - {t_{\bar n}}} )}^s}}}{\omega _{\bar m}}{{\bar \omega }_{\bar n}}} }, \]
where $x_{\bar{m}}$ are the JGL points with special parameters (depend on fractional orders in space direction), ${{\omega }_{\bar m}}$ are weights of JGL quadrature with corresponding parameters with those of JGL points, ${y}_{\bar{n}}$ are the JGR points with special parameters (depend on fractional orders in time direction), ${\hat{\omega }_{\bar n}}$ are weights of JGR quadrature with the same parameters with those of JGR points, ${{t}_{\bar n}}=\frac{T}{2}( {{y}_{\bar n}}+1 )$, $ {{{\bar{\omega }}}_{\bar n}}=\frac{T}{2}{\hat{\omega }_{\bar n}}$.

\subsection{Error Estimates}
We now analyse the error of GJF-space-time spectral Petrov-Galerkin method in different directions. Therefore, it is necessary to lead in the approximation operators as follows. 
\begin{definition}
\label{op}
Let $k,l,m,n$ be non-negative integers, satisfy $0 \leqslant k \leqslant m$ and $0 \leqslant l \leqslant n $. Define the orthogonal projector ${\Pi}_{L}^{\sigma,s}: {}_0B^{\frac{\alpha}{2}+l,\frac{\beta}{2}+k,\frac{\gamma}{2}+l,\frac{\mu}{2}+k} ( Q_T ) \mapsto V_L ( Q_T ) $ by $\forall w \in {}_0B^{\frac{\alpha}{2}+l,\frac{\beta}{2}+k,\frac{\gamma}{2}+l,\frac{\mu}{2}+k} ( Q_T ) $, ${\Pi}_{L}^{\sigma,s}w\in V_L ( Q_T ) $, such that
\[{( {w - \Pi _L^{\sigma,s}w,{\phi _L}} )_{Q_T}} = 0,\;\;\forall {\phi _L} \in {V_L}( Q_T ).\]
\end{definition}
At the beginning, let us introduce some lemmas about calculation of basis functions. Based on the basis functions constructed by GJFs, we firstly lead in the orthogonality of classical Jacobi polynomials(\textit{see Corollary 3.6 in} \cite{STW}). Suppose $\tilde \alpha ,\tilde \beta > -1 $, then for classical Jacobi polynomials $ {P_n^{( {\tilde \alpha ,\tilde \beta } )}},\;n = 0,1,2,\cdots$, where
\begin{align}
\label{Jacobi-ortho}
\int_{ - 1}^1 {P_n^{( {\tilde \alpha ,\tilde \beta } )}( x )P_{n'}^{( {\tilde \alpha ,\tilde \beta } )}( x ){\omega ^{( {\tilde \alpha ,\tilde \beta } )}}{\rm{d}}x}  = {\bar{\gamma}}_n^{({\tilde{\alpha}, \tilde{\beta}})} {\delta _{nn'}},
\end{align}
and the constant
\begin{align*}
{\bar{\gamma}}_n^{({\tilde{\alpha}, \tilde{\beta}})} = \frac{{{2^{\tilde \alpha  + \tilde \beta  + 1}}\Gamma ( {n + \tilde \alpha  + 1} )\Gamma ( {n + \tilde \beta  + 1} )}}{{( {2n + \tilde \alpha  + \tilde \beta  + 1} )n!\Gamma ( {n + \tilde \alpha  + \tilde \beta  + 1} )}}.
\end{align*}
Moreover, the integer-order derivatives of Jacobi polynomials could be calculated by(\textit{see (3.101)-(3.102) in} \cite{STW})
\begin{align}
\label{Jacobi-dk}
\frac{{\rm d}^{\bar k}}{{\rm d}x^{\bar{k}}} P_n^{( {\tilde \alpha ,\tilde \beta } )}( x ) = \frac{{\Gamma ( {n + \bar k + \tilde \alpha  + \tilde \beta  + 1} )}}{{{2^{\bar k}}\Gamma ( {n + \tilde \alpha  + \tilde \beta  + 1} )}}P_{n - \bar k}^{( {\tilde \alpha  + \bar k,\tilde \beta  + \bar k} )}( x ),\;\;n \geqslant \bar k.
\end{align}
Thus, we give the following lemmas.
\begin{lemma}
Suppose two parameters $r,\rho$, satisfying $0\leqslant r \leqslant s,\; 0 \leqslant \rho \leqslant \sigma $, it can be easily calculated that
\begin{align}
\begin{aligned}
\bar a_{ii'}^{( \rho  )}: = & {( {_{ - 1}\partial_x^\rho {\varphi _i},{\;_{ - 1}}\partial_x^\rho {\varphi _{i'}}} )_{{\omega ^{( { - \sigma  + \rho , - \sigma  + \rho } )}},\Lambda }}\\
 = & \left\{ {\begin{aligned}
{\left( {\frac{2}{{2i - 1}} \cdot \frac{{i + \sigma  - \rho }}{{i - \sigma  + \rho }} + \frac{2}{{2i + 1}}{{\left( {\frac{{i + \sigma }}{{i - \sigma }}} \right)}^2}} \right)\frac{{\Gamma ( {i - \sigma  + \rho  + 1} )}}{{\Gamma ( {i + \sigma  - \rho  + 1} )}},\;\;i = i',}\\
{ - \frac{2}{{2i + 1}} \cdot \frac{{i + \sigma }}{{i - \sigma }} \cdot \frac{{\Gamma ( {i - \sigma  + \rho  + 1} )}}{{\Gamma ( {i + \sigma  - \rho  + 1} )}},\;\;\;\;\;\;\;\;\;\;\;\;\;\;\;\;\;\;\;\;\;\;\;\;\;\;\;\;\;\;\;i = i' - 1,}\\
{ - \frac{2}{{2i' + 1}} \cdot \frac{{i' + \sigma }}{{i' - \sigma }} \cdot \frac{{\Gamma ( {i' - \sigma  + \rho  + 1} )}}{{\Gamma ( {i' + \sigma  - \rho  + 1} )}},\;\;\;\;\;\;\;\;\;\;\;\;\;\;\;\;\;\;\;\;\;\;\;\;\;\;\;\;\;i = i' + 1,}\\
{0,\;\;\;\;\;\;\;\;\;\;\;\;\;\;\;\;\;\;\;\;\;\;\;\;\;\;\;\;\;\;\;\;\;\;\;\;\;\;\;\;\;\;\;\;\;\;\;\;\;\;\;\;\;\;\;\;\;\;\;\;\;\;\;\;\;\;\;\;\;\;\;\;\;\;\;\;\;\;\;\;\;{\rm{otherwise}.}}
\end{aligned}} \right.
\end{aligned}\\
\begin{aligned}
\bar b_{jj'}^{( r )}: = {( {_0\partial_t^r{\psi _j},{{}_0}\partial_t^r{\psi _{j'}}} )_{{{\bar \omega }^{( { - s + r, - s + r} )}},I}} = {\left( {\frac{T}{2}} \right)^{2( {s - r} )}}\frac{{\Gamma ( {j - s + r} )}}{{\Gamma ( {j + s - r} )}}{\delta _{jj'}}.\;\;\;\;\;\;\;\;\;\;\;\;\;\;\;\;
\end{aligned}
\end{align}
\end{lemma}
\begin{corollary}
It holds for the special case that
\begin{align}
\begin{aligned}
\bar a_{ii'}^{( \sigma  )}: = & {( {_{ - 1}\partial_x^\sigma {\varphi _i},{_{ - 1}}\partial_x^\sigma {\varphi _{i'}}} )_{\Lambda }}\\
 = & \left\{ {\begin{aligned}
{ {\frac{2}{{2i - 1}} + \frac{2}{{2i + 1}}{{\left( {\frac{{i + \sigma }}{{i - \sigma }}} \right)}^2}} ,\;\;\;\;\;\;\;\;\;\;\;\;\;\;i = i',\;\;\;\;\;\;\;\;\;\;\;\;\;\;\;\;\;\;\;\;\;\;\;\;\;\;\;\;\;\;\;\;\;\;\;}\\
{ - \frac{2}{{2i + 1}} \cdot \frac{{i + \sigma }}{{i - \sigma }},\;\;\;\;\;\;\;\;\;\;\;\;\;\;\;\;\;\;\;\;\;\;\;\;\;\;i = i' - 1,\;\;\;\;\;\;\;\;\;\;\;\;\;\;\;\;\;\;\;\;\;\;\;\;\;\;\;\;\;\;\;\;\;\;\;}\\
{ - \frac{2}{{2i' + 1}} \cdot \frac{{i' + \sigma }}{{i' - \sigma }} ,\;\;\;\;\;\;\;\;\;\;\;\;\;\;\;\;\;\;\;\;\;\;\;\;\;i = i' + 1,\;\;\;\;\;\;\;\;\;\;\;\;\;\;\;\;\;\;\;\;\;\;\;\;\;\;\;\;\;\;\;\;\;\;\;}\\
{0,\;\;\;\;\;\;\;\;\;\;\;\;\;\;\;\;\;\;\;\;\;\;\;\;\;\;\;\;\;\;\;\;\;\;\;\;\;\;\;\;\;\;\;\;\;\;{\rm{otherwise}.\;\;\;\;\;\;\;\;\;\;\;\;\;\;\;\;\;\;\;\;\;\;\;\;\;\;\;\;\;\;\;\;\;\;\;}}
\end{aligned}} \right.
\end{aligned}\\
\begin{aligned}
\bar b_{jj'}^{( s )}: = {( {_0\partial_t^s{\psi _j},{{}_0}\partial_t^s{\psi _{j'}}} )_{I}} = {\delta _{jj'}}.\;\;\;\;\;\;\;\;\;\;\;\;\;\;\;\;\;\;\;\;\;\;\;\;\;\;\;\;\;\;\;\;\;\;\;\;\;\;\;\;\;\;\;\;\;\;\;\;\;\;\;\;\;\;\;\;\;\;\;\;\;\;\;\;\;\;\;\;\;\;
\end{aligned}
\end{align}
\end{corollary}
By \eqref{basictest-dxdt} and \eqref{Jacobi-dk}, we have
\begin{lemma}
Let $k,l$ be two non-negative integers. A direct calculation performs that
\begin{align}
\begin{aligned}
\hat a_{ii'}^{( {\sigma ,k} )} := & {( {_{ - 1}\partial_x^{\sigma  + k}{\varphi _i},{_{ - 1}}\partial_x^{\sigma  + k}{\varphi _{i'}}} )_{{\omega ^{( {k,k} )}},\Lambda }}\\
= & \left\{ {\begin{aligned}
{\left( {\frac{2}{{2i - 1}} \cdot \frac{{i - k}}{{i + k}} + \frac{2}{{2i + 1}}{{\left( {\frac{{i + \sigma }}{{i - \sigma }}} \right)}^2}} \right)\frac{{\Gamma ( {i + k + 1} )}}{{\Gamma ( {i - k + 1} )}},\;\;i = i',\;\;\;\;\;\;\;\;\;\;\;\;\;}\\
{ - \frac{2}{{2i + 1}} \cdot \frac{{i + \sigma }}{{i - \sigma }} \cdot \frac{{\Gamma ( {i + k + 1} )}}{{\Gamma ( {i - k + 1} )}},\;\;\;\;\;\;\;\;\;\;\;\;\;\;\;\;\;\;\;\;\;\;\;\;\;\;i = i' - 1,\;\;\;\;\;\;\;\;\;\;\;\;\;}\\
{ - \frac{2}{{2i' + 1}} \cdot \frac{{i' + \sigma }}{{i' - \sigma }} \cdot \frac{{\Gamma ( {i' + k + 1} )}}{{\Gamma ( {i' - k + 1} )}},\;\;\;\;\;\;\;\;\;\;\;\;\;\;\;\;\;\;\;\;\;\;\;\;i = i' + 1,\;\;\;\;\;\;\;\;\;\;\;\;\;}\\
{0,\;\;\;\;\;\;\;\;\;\;\;\;\;\;\;\;\;\;\;\;\;\;\;\;\;\;\;\;\;\;\;\;\;\;\;\;\;\;\;\;\;\;\;\;\;\;\;\;\;\;\;\;\;\;\;\;\;\;\;\;\;\;\;\;\;\;\;\;{\rm{otherwise}}.\;\;\;\;\;\;\;\;\;\;\;\;\;}
\end{aligned}} \right.
\end{aligned}\\
\begin{aligned}
\hat b_{jj'}^{( {s,l} )} : = {( {_0\partial_t^{s + l}{\psi _j},{{}_0}\partial_t^{s + l}{\psi _{j'}}} )_{{{\bar \omega }^{( {l,l} )}},I}} = {\left( {\frac{T}{2}} \right)^{2l}}\frac{{\Gamma ( {j + l} )}}{{\Gamma ( {j - l} )}}{\delta _{jj'}}.\;\;\;\;\;\;\;\;\;\;\;\;\;\;\;\;\;\;\;\;\;\;\;\;\;\;\;\;\;\;\;\;\;\;\;
\end{aligned}
\end{align}
\end{lemma}
Thanks to these lemmas above, it is of great importance to show the following lemmas.
\begin{lemma}
\label{op-op-lemma}
Suppose two fractional-order parameters $r,\rho$, which may present $0$, $\frac{\alpha}{2}$ or $\frac{\gamma}{2} $, $0$, $\frac{\beta}{2} $ or $\frac{\mu}{2}$, satisfying $0 \leqslant r \leqslant s $, $0 \leqslant \rho \leqslant \sigma $, we have
\begin{align}
\label{op-op-eq}
\begin{aligned}
& {\left\| {_0\partial _t^r{_{ - 1}}\partial _x^\rho ( {u - \Pi _L^{\sigma,s}u} )} \right\|_{{\omega ^{( { - \sigma  + \rho , - \sigma  + \rho } )}},{{\bar \omega }^{( { - s + r, - s + r} )}},Q_T}} \\
& \;\;\;\;\;\;\;\; \lesssim {{M^{\rho  - \sigma }}{\left\| {_0\partial _t^r{_{ - 1}}\partial _x^\sigma ( {u - \Pi _L^{\sigma,s}u} )} \right\|_{{{\bar \omega }^{( { - s + r, - s + r} )}},Q_T}} + {N^{r - s}}} {\left\| {_0\partial _t^s{_{ - 1}}\partial _x^\rho ( {u - \Pi _L^{\sigma,s}u} )} \right\|_{{\omega ^{( { - \sigma  + \rho , - \sigma  + \rho } )}},Q_T}} \\
& \;\;\;\;\;\;\;\;\;\;\;\;\;\;\;\;  + {M^{\rho  - \sigma }}{N^{r - s}} {\left\| {_0\partial _t^s{_{ - 1}}\partial _x^\sigma ( {u - \Pi _L^{\sigma,s}u} )} \right\|_{Q_T}}.
\end{aligned}
\end{align}
\end{lemma}
\begin{proof}
As the definition of orthogonal projector in Definition \ref{op}, we can get the consequence
\[\Pi _L^{\sigma,s}u = \sum\limits_{i = 1}^{M - 1} {\sum\limits_{j = 1}^N {{u_{ij}}{\varphi _i}( x ){\psi _j}( t )} } .\]
Compared with the exact solution \eqref{ES}, it indicates
\begin{align}
\begin{aligned}
 _0\partial _t^r {_{ - 1}}& \partial _x^\rho ( {u - \Pi _L^{\sigma,s}u} ) = \sum\limits_{i = 1}^{M - 1} {\sum\limits_{j = N + 1}^{ + \infty } {{u_{ij}}[ {_{ - 1}\partial_x^\rho {\varphi _i}( x )} ] \cdot [ {_0\partial_t^r{\psi _j}( t )} ]} } \\
& + \sum\limits_{i = M}^{ + \infty } {\sum\limits_{j = 1}^N {{u_{ij}}[ {_{ - 1}\partial_x^\rho {\varphi _i}( x )} ] \cdot [ {_0\partial_t^r{\psi _j}( t )} ]} } + \sum\limits_{i = M}^{ + \infty } {\sum\limits_{j = N + 1}^{ + \infty } {{u_{ij}}[ {_{ - 1}\partial_x^\rho {\varphi _i}( x )} ] \cdot [ {_0\partial_t^r{\psi _j}( t )} ]} } .
\end{aligned}
\end{align}
Remove terms $u_{ij} = 0$, we get
\begin{align}
\label{orginterm}
& \begin{aligned}
& \left\| {_0\partial _t^r{_{ - 1}}\partial _x^\rho ( {u - \Pi _L^{\sigma,s}u} )} \right\|_{\omega ^{( { - \sigma  + \rho , - \sigma  + \rho } )},{{\bar \omega }^{( { - s + r, - s + r} )}},Q_T}^2\\
& \;\;\;\;\;\;\;\; = \sum\limits_{j = 1}^N {\bar b_{jj}^{( r )}\sum\limits_{ i = M\atop
{u_{ij}} \ne 0}^{ + \infty } {{{\left| {{u_{ij}}} \right|}^2}( {\bar a_{ii}^{( \rho  )} + 2\bar a_{i,i + 1}^{( \rho  )} \cdot \frac{{{u_{i + 1,j}}}}{{{u_{ij}}}}} )} }  + \sum\limits_{j = N + 1}^{ + \infty } {\bar b_{jj}^{( r )}\sum\limits_{ i = 1\atop
{u_{ij}} \ne 0}^{M - 1} {{{\left| {{u_{ij}}} \right|}^2}( {\bar a_{ii}^{( \rho  )} + 2\bar a_{i,i + 1}^{( \rho  )} \cdot \frac{{{u_{i + 1,j}}}}{{{u_{ij}}}}} )} } \\
& \;\;\;\;\;\;\; \;\;\;\;\;\;\; \;\;\;\;\;\;\; + \sum\limits_{j = N}^{ + \infty } {\bar b_{jj}^{( r )}\sum\limits_{ i = M\atop
{u_{ij}} \ne 0}^{ + \infty } {{{\left| {{u_{ij}}} \right|}^2}( {\bar a_{ii}^{( \rho  )} + 2\bar a_{i,i + 1}^{( \rho  )} \cdot \frac{{{u_{i + 1,j}}}}{{{u_{ij}}}}} )} }, 
\end{aligned}\\
\label{firstterm}
& \begin{aligned}
& \left\| {_0\partial _t^s{_{ - 1}}\partial _x^\rho ( {u - \Pi _L^{\sigma,s}u} )} \right\|_{\omega ^{( { - \sigma  + \rho , - \sigma  + \rho } )}, Q_T}^2\\
& \;\;\;\;\;\;\;\; = \sum\limits_{j = 1}^N {\sum\limits_{ i = M\atop
{u_{ij}} \ne 0}^{ + \infty } {{{\left| {{u_{ij}}} \right|}^2}( {\bar a_{ii}^{( \rho  )} + 2\bar a_{i,i + 1}^{( \rho  )} \cdot \frac{{{u_{i + 1,j}}}}{{{u_{ij}}}}} )} }  + \sum\limits_{j = N + 1}^{ + \infty } {\sum\limits_{ i = 1\atop
{u_{ij}} \ne 0}^{M - 1} {{{\left| {{u_{ij}}} \right|}^2}( {\bar a_{ii}^{( \rho  )} + 2\bar a_{i,i + 1}^{( \rho  )} \cdot \frac{{{u_{i + 1,j}}}}{{{u_{ij}}}}} )} } \\
& \;\;\;\;\;\;\; \;\;\;\;\;\;\; \;\;\;\;\;\;\; + \sum\limits_{j = N}^{ + \infty } {\sum\limits_{ i = M\atop
{u_{ij}} \ne 0}^{ + \infty } {{{\left| {{u_{ij}}} \right|}^2}( {\bar a_{ii}^{( \rho  )} + 2\bar a_{i,i + 1}^{( \rho  )} \cdot \frac{{{u_{i + 1,j}}}}{{{u_{ij}}}}} )} }, 
\end{aligned}\\
\label{secondterm}
& \begin{aligned}
& \left\| {_0\partial _t^r{_{ - 1}}\partial _x^\sigma ( {u - \Pi _L^{\sigma,s}u} )} \right\|_{{\bar \omega }^{( { - s + r, - s + r} )}, Q_T}^2\\
& \;\;\;\;\;\;\;\; = \sum\limits_{j = 1}^N \bar b_{jj}^{( r )}{\sum\limits_{ i = M\atop
{u_{ij}} \ne 0}^{ + \infty } {{{\left| {{u_{ij}}} \right|}^2}( {\bar a_{ii}^{( \sigma  )} + 2\bar a_{i,i + 1}^{( \sigma  )} \cdot \frac{{{u_{i + 1,j}}}}{{{u_{ij}}}}} )} }  + \sum\limits_{j = N + 1}^{ + \infty } \bar b_{jj}^{( r )}{\sum\limits_{ i = 1\atop
{u_{ij}} \ne 0}^{M - 1} {{{\left| {{u_{ij}}} \right|}^2}( {\bar a_{ii}^{( \sigma  )} + 2\bar a_{i,i + 1}^{( \sigma  )} \cdot \frac{{{u_{i + 1,j}}}}{{{u_{ij}}}}} )} } \\
& \;\;\;\;\;\;\; \;\;\;\;\;\;\; \;\;\;\;\;\;\; + \sum\limits_{j = N}^{ + \infty }\bar b_{jj}^{( r )} {\sum\limits_{ i = M\atop
{u_{ij}} \ne 0}^{ + \infty } {{{\left| {{u_{ij}}} \right|}^2}( {\bar a_{ii}^{( \sigma  )} + 2\bar a_{i,i + 1}^{( \sigma  )} \cdot \frac{{{u_{i + 1,j}}}}{{{u_{ij}}}}} )} },
\end{aligned}\\
\label{thirdterm}
& \begin{aligned}
& \left\| {_0\partial _t^s{_{ - 1}}\partial _x^\sigma ( {u - \Pi _L^{\sigma,s}u} )} \right\|_{Q_T}^2\\
& \;\;\;\;\;\;\;\; = \sum\limits_{j = 1}^N {\sum\limits_{ i = M\atop
{u_{ij}} \ne 0}^{ + \infty } {{{\left| {{u_{ij}}} \right|}^2}( {\bar a_{ii}^{( \sigma  )} + 2\bar a_{i,i + 1}^{( \sigma  )} \cdot \frac{{{u_{i + 1,j}}}}{{{u_{ij}}}}} )} }  + \sum\limits_{j = N + 1}^{ + \infty } {\sum\limits_{ i = 1\atop{u_{ij}} \ne 0}^{M - 1} {{{\left| {{u_{ij}}} \right|}^2}( {\bar a_{ii}^{( \sigma  )} + 2\bar a_{i,i + 1}^{( \sigma  )} \cdot \frac{{{u_{i + 1,j}}}}{{{u_{ij}}}}} )} } \\
& \;\;\;\;\;\;\; \;\;\;\;\;\;\; \;\;\;\;\;\;\; + \sum\limits_{j = N}^{ + \infty } {\sum\limits_{ i = M\atop
{u_{ij}} \ne 0}^{ + \infty } {{{\left| {{u_{ij}}} \right|}^2}( {\bar a_{ii}^{( \sigma  )} + 2\bar a_{i,i + 1}^{( \sigma  )} \cdot \frac{{{u_{i + 1,j}}}}{{{u_{ij}}}}} )} }.
\end{aligned}
\end{align}
Denote the term
\begin{align*}
\bar a_{ii}^{( \rho  )} + 2\bar a_{i,i + 1}^{( \rho  )} \cdot \frac{{{u_{i + 1,j}}}}{{{u_{ij}}}} = & \frac{{\Gamma ( {i - \sigma  + \rho  + 1} )}}{{\Gamma ( {i + \sigma  - \rho  + 1} )}}\left[ {\frac{2}{{2i - 1}} \cdot \frac{{i + \sigma  - \rho }}{{i - \sigma  + \rho }} + \frac{2}{{2i + 1}} \cdot \frac{{i + \sigma }}{{i - \sigma }}\left( {\frac{{i + \sigma }}{{i - \sigma }} - 2\frac{u_{i + 1,j}}{u_{ij}}} \right)} \right]\\
: = & \frac{{\Gamma ( {i - \sigma  + \rho  + 1} )}}{{\Gamma ( {i + \sigma  - \rho  + 1} )}} \cdot \bar d_{ij}^{( \rho  )}, 
\end{align*}
where
$${\bar{d}}_{ij}^{( \rho  )} \leqslant \frac{i+\sigma-\rho}{i-\sigma+\rho} {\bar{d}}_{ij}^{( \sigma  )}. $$
Otherwise, we separately compare the three terms in the right hand of \eqref{orginterm} with \eqref{firstterm} -- \eqref{thirdterm}, 
\begin{align*}
\begin{aligned}
&  \left\| {_0\partial _t^r{_{ - 1}}\partial _x^\rho ( {u - \Pi _L^{\sigma,s}u} )} \right\|_{{\omega ^{( { - \sigma  + \rho , - \sigma  + \rho } )}},{{\bar \omega }^{( { - s + r, - s + r} )}},Q_T}^2 \\
& \;\;\;\;\;\;\;\;= {\left( \frac{T}{2} \right)^{2( {s - r} )}}\sum\limits_{j = 1}^N {\frac{{\Gamma ( {j - s + r} )}}{{\Gamma ( {j + s - r} )}}\sum\limits_{\scriptstyle i = M\atop
\scriptstyle{u_{ij}} \ne 0}^{ + \infty } {\frac{{\Gamma ( {i - \sigma  + \rho  + 1} )}}{{\Gamma ( {i + \sigma  - \rho  + 1} )}} {\left| {{u_{ij}}} \right|}^2} \bar d_{ij}^{( \rho  )} } \\
& \;\;\;\;\;\;\;\;\;\;\;\;\;\; + {\left( \frac{T}{2} \right)^{2( {s - r} )}}\sum\limits_{j = N + 1}^{ + \infty } {\frac{{\Gamma ( {j - s + r} )}}{{\Gamma ( {j + s - r} )}}\sum\limits_{\scriptstyle i = 1\atop
\scriptstyle{u_{ij}} \ne 0}^{M - 1} {\frac{{\Gamma ( {i - \sigma  + \rho  + 1} )}}{{\Gamma ( {i + \sigma  - \rho  + 1} )}} {{\left| {{u_{ij}}} \right|}^2}\bar d_{ij}^{( \rho  )}} } \\
& \;\;\;\;\;\;\;\;\;\;\;\;\;\; + {\left( \frac{T}{2} \right)^{2( {s - r} )}}\sum\limits_{j = N + 1}^{ + \infty } {\frac{{\Gamma ( {j - s + r} )}}{{\Gamma ( {j + s - r} )}}\sum\limits_{\scriptstyle i = M\atop
\scriptstyle{u_{ij}} \ne 0}^{ + \infty } {\frac{{\Gamma ( {i - \sigma  + \rho  + 1} )}}{{\Gamma ( {i + \sigma  - \rho  + 1} )}} {{\left| {{u_{ij}}} \right|}^2} \bar d_{ij}^{( \rho  )}} } \\
& \;\;\;\;\;\;\;\;\lesssim {\left( \frac{T}{2} \right)^{2( {s - r} )}}\frac{{\Gamma ( {1 - s + r} )}}{{\Gamma ( {1 + s - r} )}}\frac{{\Gamma ( {M - \sigma  + \rho  + 1} )}}{{\Gamma ( {M + \sigma  - \rho  + 1} )}}\sum\limits_{j = 1}^N {\sum\limits_{\scriptstyle i = M\atop
\scriptstyle{u_{ij}} \ne 0}^{ + \infty } {{{\left| {{u_{ij}}} \right|}^2}\bar d_{ij}^{( \sigma  )}} } \\
& \;\;\;\;\;\;\;\;\;\;\;\;\;\; + {\left( \frac{T}{2} \right)^{2( {s - r} )}}\frac{{\Gamma ( {N + 1 - s + r} )}}{{\Gamma ( {N + 1 + s - r} )}}\frac{{\Gamma ( {2 - \sigma  + \rho } )}}{{\Gamma ( {2 + \sigma  - \rho } )}}\sum\limits_{j = N + 1}^{ + \infty } {\sum\limits_{\scriptstyle i = 1\atop
\scriptstyle{u_{ij}} \ne 0}^{M - 1} {{{\left| {{u_{ij}}} \right|}^2}\bar d_{ij}^{( \rho  )}} } \\
& \;\;\;\;\;\;\;\;\;\;\;\;\;\; + {\left( \frac{T}{2} \right)^{2( {s - r} )}}\frac{{\Gamma ( {N + 1 - s + r} )}}{{\Gamma ( {N + 1 + s - r} )}}\frac{{\Gamma ( {M - \sigma  + \rho  + 1} )}}{{\Gamma ( {M + \sigma  - \rho  + 1} )}}\sum\limits_{j = N + 1}^{ + \infty } {\sum\limits_{\scriptstyle i = M\atop
\scriptstyle{u_{ij}} \ne 0}^{ + \infty } {{{\left| {{u_{ij}}} \right|}^2}\bar d_{ij}^{( \sigma  )}} }\\
& \;\;\;\;\;\;\;\;\lesssim {M^{2\rho  - 2\sigma } }\left\| {_0\partial _t^r{_{ - 1}}\partial _x^\sigma ( {u - \Pi _L^{\sigma,s}u} )} \right\|_{{{\bar \omega }^{( { - s + r, - s + r} )}}, Q_T}^2 + {N^{2r - 2s }}\left\| {_0\partial _t^s{_{ - 1}}\partial _x^\rho ( {u - \Pi _L^{\sigma,s}u} )} \right\|_{{\omega ^{( { - \sigma  + \rho , - \sigma  + \rho } )}},Q_T}^2 \\
& \;\;\;\;\;\;\;\;\;\;\;\;\;\; + {M^{2\rho  - 2\sigma }}{N^{2r - 2s}}\left\| {_0\partial _t^s{_{ - 1}}\partial _x^\sigma ( {u - \Pi _L^{\sigma,s}u} )} \right\|_{Q_T}^2,
\end{aligned}
\end{align*}
which leads to \eqref{op-op-eq}.
\end{proof}
\begin{corollary}
Use the result of Lemma \ref{op-op-lemma} for twice, we get the following result:
\begin{align}
{\left\| {_0\partial _t^r{_{ - 1}}\partial _x^\rho ( {u - \Pi _L^{\sigma,s}u} )} \right\|_{{\omega ^{( { - \sigma  + \rho , - \sigma  + \rho } )}},{{\bar \omega }^{( { - s + r, - s + r} )}},Q_T}}  \lesssim  {M^{\rho  - \sigma }}{N^{r - s}} {\left\| {_0\partial _t^s{_{ - 1}}\partial _x^\sigma ( {u - \Pi _L^{\sigma,s}u} )} \right\|_{Q_T}}.
\end{align}
\end{corollary}
\begin{lemma}
\label{op-dmn-lemma}
Let $k,l,m,n $ be four non-negative integers, satisfying $0 \leqslant k \leqslant m \leqslant M $, $0 \leqslant l \leqslant n \leqslant N $. We arrive at
\begin{align}
\label{op-dmn-eq}
\begin{aligned}
& {\left\| {_0\partial _t^{s + l}{_{ - 1}}\partial _x^{\sigma  + k}( {u - \Pi _L^{\sigma,s}u} )} \right\|_{{\omega ^{( {k,k} )}},{{\bar \omega }^{( {l,l} )}},Q_T}} \lesssim {M^{k - m}}{\left\| {_0\partial _t^{s + l}{_{ - 1}}\partial _x^{\sigma  + m}u} \right\|_{{\omega ^{( {m,m} )}}, {{\bar \omega }^{( {l,l} )}}, Q_T}} \\
& \;\;\;\;\;\;\;\;\;\;\;\; + {N^{l - n}}{\left\| {_0\partial _t^{s + n}{_{ - 1}}\partial _x^{\sigma  + k}u} \right\|_{{\omega ^{( {k,k} )}}, {{\bar \omega }^{( {n,n} )}},Q_T}}+ {M^{k - m}}{N^{l - n}}{\left\| {_0\partial _t^{s + n}{_{ - 1}}\partial _x^{\sigma  + m}u} \right\|_{{\omega ^{( {m,m} )}},{{\bar \omega }^{( {n,n} )}},Q_T}}.
\end{aligned}
\end{align}
\end{lemma}
\begin{proof}
By an argument similar to the proof of Lemma \ref{op-op-lemma}, we have
\begin{align}
\begin{aligned}
& _0\partial _t^{s + l}{_{ - 1}}\partial _x^{\sigma  + k}( {u - \Pi _L^{\sigma,s}u} ) = \sum\limits_{i = k}^{M - 1} {\sum\limits_{j = N + 1}^{ + \infty } {{u_{ij}}[ {_{ - 1}\partial_x^{\sigma  + k}{\varphi _i}( x )} ] \cdot [ {_0\partial_t^{s + l}{\psi _j}( t )} ]} } \\
& \;\;\;\;\;\;\ + \sum\limits_{i = M}^{ + \infty } {\sum\limits_{j = l}^N {{u_{ij}}[ {_{ - 1}\partial_x^{\sigma  + k}{\varphi _i}( x )} ] \cdot [ {_0\partial_t^{s + l}{\psi _j}( t )} ]} }  + \sum\limits_{i = M}^{ + \infty } {\sum\limits_{j = N + 1}^{ + \infty } {{u_{ij}}[ {_{ - 1}\partial_x^{\sigma  + k}{\varphi _i}( x )} ] \cdot [ {_0\partial_t^{s + l}{\psi _j}( t )} ]} }.
\end{aligned}
\end{align}
Ignore terms $u_{ij} = 0 $, we can get
\begin{align}
\begin{aligned}
&  \left\| {_0\partial _t^{s + l}{\;_{ - 1}}\partial _x^{\sigma  + k}( {u - \Pi _L^{\sigma,s}u} )} \right\|_{{\omega ^{( {k,k} )}},{{\bar \omega }^{( {l,l} )}},Q_T}^2\\
& \;\;\;\;\;\;\;\;= \sum\limits_{j = 1}^N {\hat b_{jj}^{( {s,l} )}\sum\limits_{\scriptstyle i = M \atop
\scriptstyle{u_{ij}} \ne 0}^{ + \infty } {{{\left| {{u_{ij}}} \right|}^2}( {\hat a_{ii}^{( {\sigma ,k} )} + 2\hat a_{i,i + 1}^{( {\sigma ,k} )} \cdot \frac{{{u_{i + 1,j}}}}{{{u_{ij}}}}} )} }  + \sum\limits_{j = N + 1}^{ + \infty } {\hat b_{jj}^{( {s,l} )}\sum\limits_{\scriptstyle i = 1\atop
\scriptstyle{u_{ij}} \ne 0}^{M - 1} {{{\left| {{u_{ij}}} \right|}^2}( {\hat a_{ii}^{( {\sigma ,k} )} + 2\hat a_{i,i + 1}^{( {\sigma ,k} )} \cdot \frac{{{u_{i + 1,j}}}}{{{u_{ij}}}}} )} } \\
& \;\;\;\;\;\;\;\;\;\;\;\;\; + \sum\limits_{j = N + 1}^{ + \infty } {\hat b_{jj}^{( {s,l} )}\sum\limits_{\scriptstyle i = M \atop
\scriptstyle{u_{ij}} \ne 0}^{ + \infty } {{{\left| {{u_{ij}}} \right|}^2}( {\hat a_{ii}^{( {\sigma ,k} )} + 2\hat a_{i,i + 1}^{( {\sigma ,k} )} \cdot \frac{{{u_{i + 1,j}}}}{{{u_{ij}}}}} )} }.
\end{aligned}
\end{align}
Furthermore, for $m=k,k+1,k+2,\cdots,M $ and $n=l,l+1,l+2,\cdots,N $, we take
\begin{align}
\left\| {_0\partial _t^{s + n}{_{ - 1}}\partial _x^{\sigma  + m}u} \right\|_{{\omega ^{( {m,m} )}},{{\bar \omega }^{( {n,n} )}},Q_T}^2 = \sum\limits_{j = n + 1}^{ + \infty } {\hat b_{jj}^{( {s,n} )}\sum\limits_{\scriptstyle i = m\atop
\scriptstyle{u_{ij}} \ne 0}^{ + \infty } {{{\left| {{u_{ij}}} \right|}^2}( {\hat a_{ii}^{( {\sigma ,m} )} + 2\hat a_{i,i + 1}^{( {\sigma ,m} )} \cdot \frac{{{u_{i + 1,j}}}}{{{u_{ij}}}}} )} }.
\end{align}
Denote
\begin{align*}
\hat a_{ii}^{( {\sigma ,k} )} + 2\hat a_{i,i + 1}^{( {\sigma ,k} )} \cdot \frac{{{u_{i + 1,j}}}}{{{u_{ij}}}} = & \frac{{\Gamma ( {i + k + 1} )}}{{\Gamma ( {i - k + 1} )}}\left( {\frac{2}{{2i - 1}} \cdot \frac{{i - k}}{{i + k}} + \frac{2}{{2i + 1}} \cdot \frac{{i + \sigma }}{{i - \sigma }}\left( {\frac{{i + \sigma }}{{i - \sigma }} - 2\frac{{{u_{i + 1,j}}}}{{{u_{ij}}}}} \right)} \right)\\
: = & \frac{{\Gamma ( {i + k + 1} )}}{{\Gamma ( {i - k + 1} )}}\hat d_{ij}^{( k )}.
\end{align*}
It performs that ${\hat{d}}_{ij}^{( k )} \lesssim {\hat{d}}_{ij}^{( m )} $, so
\begin{align*}
&  \left\| {_{\rm{0}}\partial _t^{s + l}{_{ - 1}}\partial _x^{\sigma  + k}( {u - \Pi _L^{\sigma,s}u} )} \right\|_{{\omega ^{( {k,k} )}},{{\bar \omega }^{( {l,l} )}},Q_T}^{\rm{2}}\\
& \;\;\;\;\;\;\;\;= \sum\limits_{j = l}^N {{\left( \frac{T}{2} \right)^{2l}}\frac{{\Gamma ( {j + l} )}}{{\Gamma ( {j - l} )}}\sum\limits_{\scriptstyle i = M\atop
\scriptstyle{u_{ij}} \ne 0}^{ + \infty } {\frac{{\Gamma ( {i + k + 1} )}}{{\Gamma ( {i - k + 1} )}}\hat d_{ij}^{( k )} {{\left| {{u_{ij}}} \right|}^2}} } \\
& \;\;\;\;\;\; \;\;\;\;\;\;\;\;\;\;  + \sum\limits_{j = N + 1}^{ + \infty } {{\left( \frac{T}{2} \right)^{2l}}\frac{{\Gamma ( {j + l} )}}{{\Gamma ( {j - l} )}}\sum\limits_{\scriptstyle i = k\atop
\scriptstyle{u_{ij}} \ne 0}^{M - 1} {\frac{{\Gamma ( {i + k + 1} )}}{{\Gamma ( {i - k + 1} )}}\hat d_{ij}^{( k )} {{\left| {{u_{ij}}} \right|}^2}} } \\
& \;\;\;\;\;\; \;\;\;\;\;\;\;\;\;\;  + \sum\limits_{j = N + 1}^{ + \infty } {{\left( \frac{T}{2} \right)^{2l}}\frac{{\Gamma ( {j + l} )}}{{\Gamma ( {j - l} )}}\sum\limits_{\scriptstyle   i = M\atop
\scriptstyle{u_{ij}} \ne 0}^{ + \infty } {\frac{{\Gamma ( {i + k + 1} )}}{{\Gamma ( {i - k + 1} )}}\hat d_{ij}^{( k )} {{\left| {{u_{ij}}} \right|}^2}} } \\
& \;\;\;\;\;\;\;\;\leqslant \frac{{\Gamma ( {M + k + 1} )\Gamma ( {M - m + 1} )}}{{\Gamma ( {M - k + 1} )\Gamma ( {M + m + 1} )}} \sum\limits_{j = l}^{ + \infty } {{\left( \frac{T}{2} \right)^{2l}}\frac{{\Gamma ( {j + l} )}}{{\Gamma ( {j - l} )}}\sum\limits_{\scriptstyle i = M\atop
\scriptstyle{u_{ij}} \ne 0}^{ + \infty } {\frac{{\Gamma ( {i + m + 1} )}}{{\Gamma ( {i - m + 1} )}}\hat d_{ij}^{( k )} {{\left| {{u_{ij}}} \right|}^2}} }
\end{align*}
\begin{align*}
& \;\;\;\;\;\; \;\;\;\;\;\;\;\;\;\; + {\left( \frac{T}{2} \right)^{2l - 2n}}\frac{{\Gamma ( {N + 1 + l} )\Gamma ( {N + 1 - n} )}}{{\Gamma ( {N + 1 - l} )\Gamma ( {N + 1 + n} )}} \\
& \;\;\;\;\;\; \;\;\;\;\;\; \;\;\;\;\;\;\;\;\;\; \cdot \sum\limits_{j = N + 1}^{ + \infty } {{{\left( \frac{T}{2} \right)}^{2n}}\frac{{\Gamma ( {j + n} )}}{{\Gamma ( {j - n} )}}\sum\limits_{\scriptstyle i = k\atop
\scriptstyle{u_{ij}} \ne 0}^{ + \infty } {\frac{{\Gamma ( {i + k + 1} )}}{{\Gamma ( {i - k + 1} )}}\hat d_{ij}^{( k )} {{\left| {{u_{ij}}} \right|}^2}} } \\
& \;\;\;\;\;\; \;\;\;\;\;\;\;\;\;\; + {\left( \frac{T}{2} \right)^{2l - 2n}}\frac{{\Gamma ( {N + 1 + l} )\Gamma ( {N + 1 - n} )}}{{\Gamma ( {N + 1 - l} )\Gamma ( {N + 1 + n} )}}\frac{{\Gamma ( {M + k + 1} )\Gamma ( {M - m + 1} )}}{{\Gamma ( {M - k + 1} )\Gamma ( {M + m + 1} )}} \\
& \;\;\;\;\;\; \;\;\;\;\;\; \;\;\;\;\;\;\;\;\;\; \cdot \sum\limits_{j = N + 1}^{ + \infty } {{{\left( \frac{T}{2} \right)}^{2n}}\frac{{\Gamma ( {j + l} )}}{{\Gamma ( {j - l} )}}\sum\limits_{\scriptstyle i = M\atop
\scriptstyle{u_{ij}} \ne 0}^{ + \infty } {\frac{{\Gamma ( {i + k + 1} )}}{{\Gamma ( {i - k + 1} )}}\hat d_{ij}^{( k )} {{\left| {{u_{ij}}} \right|}^2}} } \\
& \;\;\;\;\;\;\;\;\lesssim {M^{2k - 2m}}\left\| {_{\rm{0}}\partial _t^{s + l}{_{ - 1}}\partial _x^{\sigma  + m}u} \right\|_{{\omega ^{( {m,m} )}},{{\bar \omega }^{( {l,l} )}},Q_T}^{\rm{2}} + {N^{2l - 2n}}\left\| {_{\rm{0}}\partial _t^{s + n}{_{ - 1}}\partial _x^{\sigma  + k}u} \right\|_{{\omega ^{( {k,k} )}},{{\bar \omega }^{( {n,n} )}},Q_T}^{\rm{2}} \\
& \;\;\;\;\;\;\;\;\;\;\;\;\;\;\;\;  + {M^{2k - 2m}}{N^{2l - 2n}}\left\| {_{\rm{0}}\partial _t^{s + n}{_{ - 1}}\partial _x^{\sigma  + m}u} \right\|_{{\omega ^{( {m,m} )}},{{\bar \omega }^{( {n,n} )}},Q_T}^{\rm{2}}.
\end{align*}
This ends the proof.
\end{proof}
Based on the two lemmas above, it holds for the following direct conclusions.
\begin{lemma}
\label{mainlemma}
Denote two fractional-order parameters $r,\rho$, satisfying $0 \leqslant r \leqslant s $, $0 \leqslant \rho \leqslant \sigma $, and regularities $k,l,m,n $ be four non-negative integers, satisfying $0 \leqslant k \leqslant m \leqslant M $, $0 \leqslant l \leqslant n \leqslant N $, it appears
\begin{align}
\begin{aligned}
& {\left\| {_0\partial _t^{r+l}{_{ - 1}}\partial _x^{\rho+k} ( {u - \Pi _L^{\sigma,s}u} )} \right\|_{{\omega ^{( { - \sigma  + \rho + k, - \sigma  + \rho + k} )}},{{\bar \omega }^{( { - s + r + l, - s + r + l} )}},Q_T}} \\
& \;\;\;\;\;\;\;\; \lesssim {M^{\rho - \sigma + k - m}}{N^{r-s}}{\left\| {_0\partial _t^{s + l}{_{ - 1}}\partial _x^{\sigma  + m}u} \right\|_{{\omega ^{( {m,m} )}},{{\bar \omega }^{( {l,l} )}},Q_T}} + {M^{\rho - \sigma}}{N^{r - s + l - n}}{\left\| {_0\partial _t^{s + n}{_{ - 1}}\partial _x^{\sigma  + k}u} \right\|_{{\omega ^{( {k,k} )}}, {{\bar \omega }^{( {n,n} )}},Q_T}}\\
& \;\;\;\;\;\;\;\;\;\;\;\; + {M^{\rho - \sigma + k - m}}{N^{r - s + l - n}}{\left\| {_0\partial _t^{s + n}{_{ - 1}}\partial _x^{\sigma  + m}u} \right\|_{{\omega ^{( {m,m} )}},{{\bar \omega }^{( {n,n} )}},Q_T}}.
\end{aligned}
\end{align}
\end{lemma}

\begin{theorem}
Suppose fractional orders $\alpha,\gamma \in (0,1), \beta,\mu \in (1,2)$, $k,l,m,n$ are non-negative integers, satisfying $0 \leqslant k \leqslant m, 0 \leqslant l \leqslant n $. Let $u, u_L$ be respectively the solutions of \eqref{3-1} and \eqref{4-5}. Denote
\begin{align*}
s = \max\left\lbrace{\frac{\alpha }{2},\frac{\gamma }{2}}\right\rbrace, \sigma = \max\left\lbrace{\frac{\beta }{2},\frac{\mu }{2}}\right\rbrace.
\end{align*}
If
\begin{align*}
u \in {}_0H^{s+l}({I;H^{\sigma+m}(\Lambda)}) \cap H^{s+n}({I;H_0^{\sigma+k}(\Lambda)}) \cap H^{s+n}({I;H^{\sigma+m}(\Lambda)}),
\end{align*}
then we have
\begin{align}
\label{mainerrorestimates}
\begin{aligned}
& \left\| u - u_L \right\|_{{B^{\frac{\alpha }{2}+l,\frac{\beta }{2}+k,\frac{\gamma }{2}+l,\frac{\mu }{2}+k}}( Q_T )} \lesssim {M^{- \sigma + k - m}}{N^{\frac{\alpha }{2}-s}}{\left\| {_0\partial _t^{s+l}{_{ - 1}}\partial _x^{\sigma  + m}u} \right\|_{{\omega ^{( {m,m} )}},{{\bar \omega }^{( {l,l} )}},Q_T}} \\
& \;\;+ {M^{\frac{\beta }{2} - \sigma}}{N^{- s + l - n}}{\left\| {_0\partial _t^{s + n}{_{ - 1}}\partial _x^{\sigma +k}u} \right\|_{{\omega ^{( {k,k} )}},{{\bar \omega }^{( {n,n} )}},Q_T}} + {M^{\frac{\mu }{2} - \sigma + k - m}}{N^{\frac{\gamma }{2} - s + l - n}}{\left\| {_0\partial _t^{s + n}{_{ - 1}}\partial _x^{\sigma  + m}u} \right\|_{{\omega ^{( {m,m} )}},{{\bar \omega }^{( {n,n} )}},Q_T}}.
\end{aligned}
\end{align}
\end{theorem}
\begin{proof}
Thanks to the standard conclusion of error estimation, we could get the fomula
\begin{align}
\label{standarderror}
\left\| u - u_L \right\|_{{B^{\frac{\alpha }{2}+l,\frac{\beta }{2}+k,\frac{\gamma }{2}+l,\frac{\mu }{2}+k}}( Q_T )} \leqslant \inf_{\phi_L \in V_L} \left\| u - \phi_L \right\|_{{B^{\frac{\alpha }{2}+l,\frac{\beta }{2}+k,\frac{\gamma }{2}+l,\frac{\mu }{2}+k}}( Q_T )}.
\end{align}
Taking $\phi_L = \Pi _L^{\sigma,s}u$ in the right hand of \eqref{standarderror}, we can get the conclusion by Lemma \ref{mainlemma}.
\end{proof}

\section{Numerical Results}

In this section, we aim to verify the convergence of the problem \eqref{1-1}-\eqref{1-3} via several numerical results. More importantly, our goal here is to get a both fast and accurate solution. We performed our computations by using Matlab 2018b software on an Intel(R) Core(TM) i7, 2.9GHz CPU machine with 16 Gbyte of memory.
\subsection{Test Problem 1}

We consider the problem \eqref{1-1} with exact solution
\begin{align}
\label{TP1}
u({x,t})=( 1 - x ){( 1+x) }^{\sigma + \eta }E_{1,\sigma + \eta +1}(1+x)\cdot{\left( \frac{2t}{T} \right)}^{s+\theta}E_{1,s+\theta+1}\left(\frac{2t}{T}\right),
\end{align}
where $E_{\bar{a},\bar{b}}(z)$ present the famous Mittag-Leffler (ML) function with two parameters $\bar{a},\bar{b} \in \mathbb{C}$, defined by means of the series expansion
\begin{align*}
E_{\bar{a},\bar{b}}(z)=\sum^{\infty}_{\bar{k}=0}\frac{z^{\bar{k}}}{\Gamma(\bar{a}\bar{k}+\bar{b})}.
\end{align*}
We choose the fractional orders $\alpha = 0.5,\ \beta = 1.2,\ \gamma = 0.2,\ \mu = 1.8 $, which means
$$\sigma = \max\left\lbrace\frac{\beta}{2}, \frac{\mu}{2} \right\rbrace = 0.9,\ s = \max\left\lbrace\frac{\alpha}{2}, \frac{\gamma}{2} \right\rbrace = 0.25, $$
and set $\varepsilon = 1,\ \eta = \theta = 4 $. The exact solution \eqref{TP1} satisfies the initial and boundary conditions \eqref{1-2} and \eqref{1-3}.
We get the results via  Figure 1 and Figure 2.

\begin{figure}
\label{Fig1}
\begin{center}
\caption{$L^2$-norm error estimate varies from $N$ with different parameter $M$ for Test Problem 1.}
\includegraphics[width=10cm]{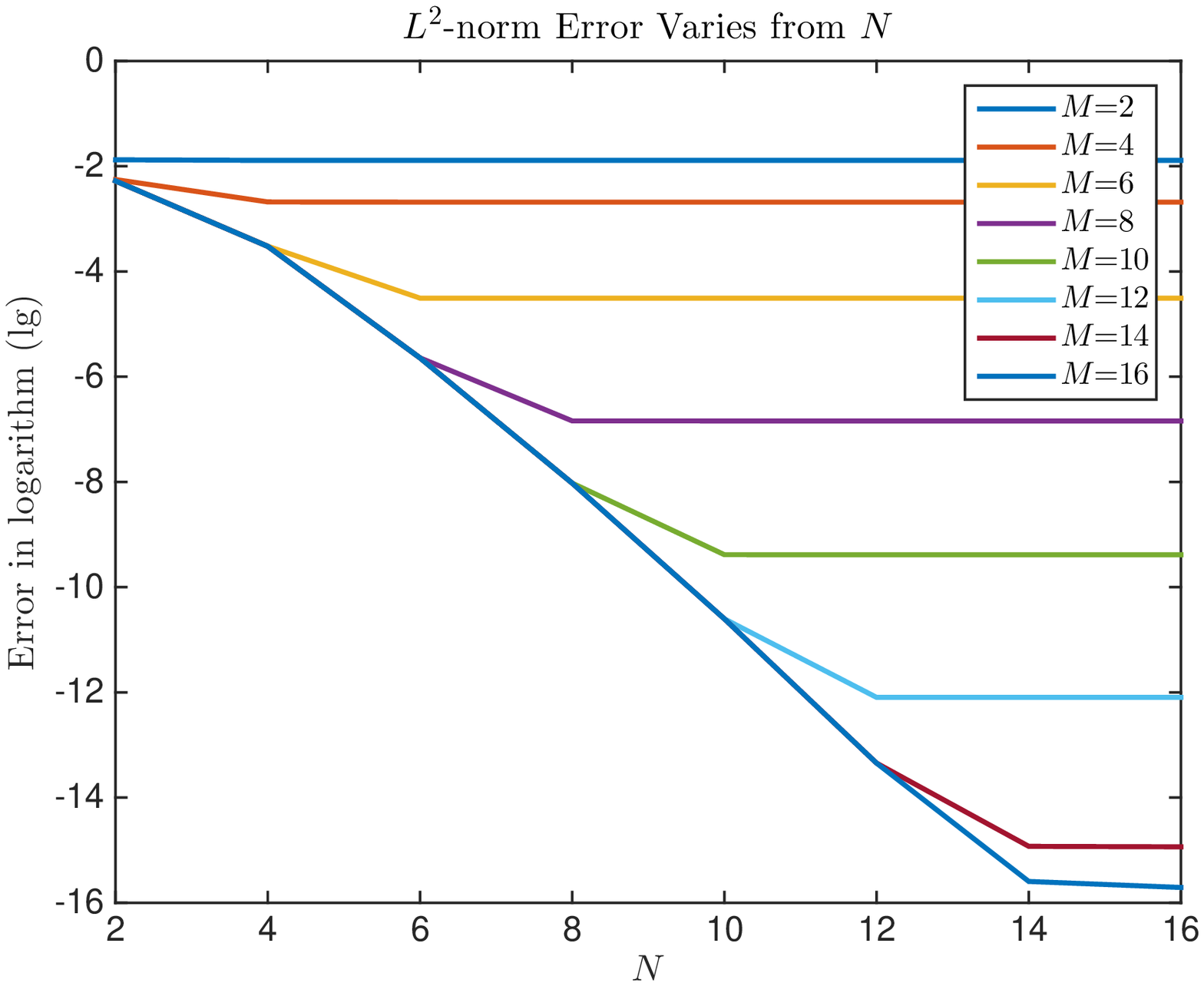}
\end{center}
\end{figure}

\begin{figure}
\label{Fig2}
\begin{center}
\caption{$L^2$-norm error estimate varies from $M$ with  different parameter $N$ for Test Problem 1.}
\includegraphics[width=10cm]{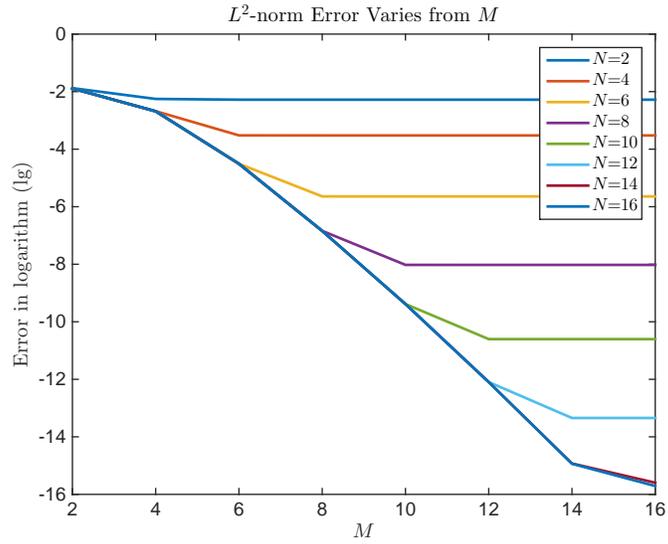}
\end{center}
\end{figure}

We see that the error in logarithm dropped rapidly in Figure 1. In fact, the convergence is also perfect while $\alpha, \gamma=0.01$ or $\alpha, \gamma = 0.99$, and $\beta, \mu=1.01$ or $\beta, \mu = 1.99$, that is, fractional space-time spectral method is available to all fractional orders in those ranges. Compared with difference in time or space, it is better for fractional space-time spectral method to solve these traditional diffusion equations because of keeping the global characteristics. 

\subsection{Test Problem 2}
This part will show the relationship with error estimates and regularity of exact solutions. Here, we choose fractional orders  $\alpha = 0.5,\beta = 1.2,\gamma = 0.2,\mu = 1.8 $, and regularities $\theta = \eta = 8$, we show the error estimates of norm ${}_0H^{s+l}({I;H_0^{\sigma}({\Lambda})})$ vary from $N$ with different parameters $l$ in Figure 3. Similarly, we always get those of norm ${}_0H^s({I;H_0^{\sigma+k}({\Lambda})})$ vary from $M$ with different parameters $k$ in Figure 4.
\begin{figure}
\label{Fig3}
\begin{center}
\caption{${}_0H^{s+l}({I;H^{\sigma}_0({\Lambda})})$-norm error estimate varies from $N$ with  different parameter $l$ for Test Problem 2.}
\includegraphics[width=10cm]{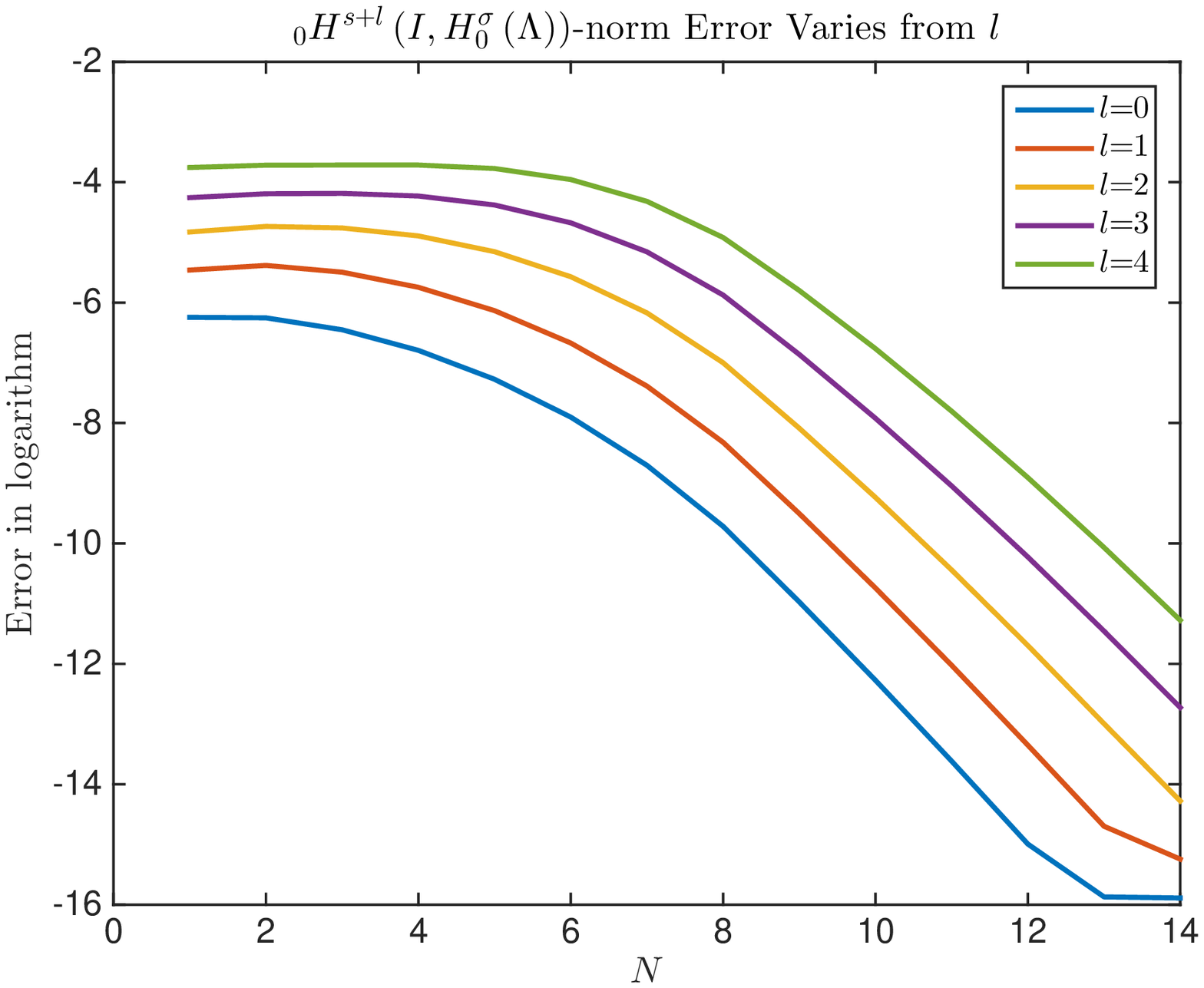}
\end{center}
\end{figure}

\begin{figure}
\label{Fig4}
\begin{center}
\caption{${}_0H^{s}({I;H^{\sigma+k}_0({\Lambda})})$-norm error estimate varies from $M$ with  different parameter $k$ for Test Problem 2.}
\includegraphics[width=10cm]{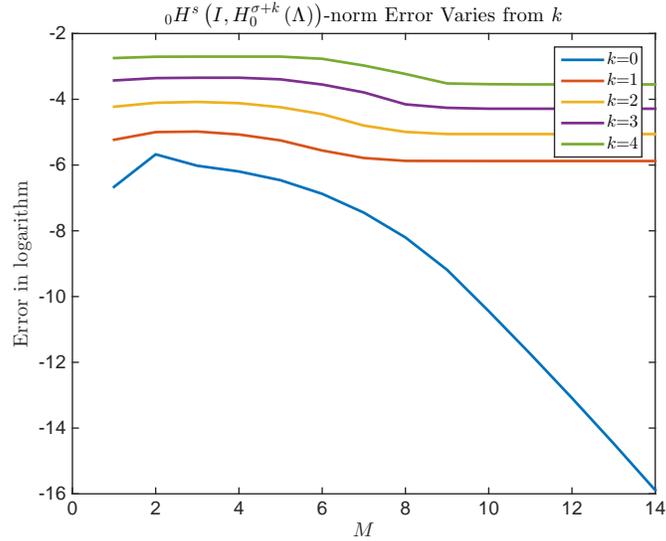}
\end{center}
\end{figure}
As expected, the convergence orders decreased with $k$ or $l$ increased. What's more, these error estimates satisfy the spectral convergence.

\subsection{Test Problem 3}
As for this part, we need to examine how the viscosity term influences the space-time diffusion. Suppose
\begin{align*}
f ( x,t ) = {(1+x)}^{\sigma+\eta}{\left( \frac{2t}{T} \right)}^{s+\theta}{\left\lbrace{(1-x)E_{1,\sigma+\eta+1}{(1+x)}\left[{\left( \frac{2t}{T} \right)}^{-\alpha}E_{1,s+\theta-\alpha+1}{\left( \frac{2t}{T} \right)}+E_{1,s+\theta+1}{\left( \frac{2t}{T} \right)}\right]}\right\rbrace} \\
 - {(1+x)}^{\sigma+\eta-\beta}{\left( \frac{2t}{T} \right)}^{s+\theta}{E_{1,s+\theta+1}{\left( \frac{2t}{T} \right)}\left[{(1-x)}E_{1,\sigma+\eta-\beta+1}{(1+x)}-\beta{(1+x)}E_{1,\sigma+\eta-\beta+2}{(1+x)}\right]},
\end{align*}
where the exact solution is \eqref{TP1} while $\varepsilon = 0$. We change the parameter $\varepsilon$ from 0 to 1. Denote fractional orders  $\alpha = 0.5,\beta = 1.5,\gamma = 0.5,\mu = 1.5 $, and  $\theta = \eta = 4$, we get the graphs of numerical solutions depend on the precision in Figure 5.

Based on Figure 5, we could say that, with the parameter $\varepsilon$ increasing from 0 to 1, the fractional diffusion becomes slow. In fact, we verify by numerical tests, that the term $_0\partial_t^{\gamma} {}_{-1}\partial_x^{\mu} u(x,t) $ presents viscosity term in space-time fractional diffusion. 

\section{Concluding remarks}

This paper proposes a new space-time spectral method solving a kind of space-time fractional reaction-diffusion equations with viscosity terms via GJFs. The existence and uniqueness of the problem are proved by the Babu\v{s}ka Lax-Milgram lemma. We get good results with constructed spectral basis functions in both space and time directions. Based on the spectral basis functions, the coefficient matrices seem sparse. At the same time, the method has been calculated fast and efficiently. By means of numerical results, we can simply solve the problems with Riemann-Liouville fractional derivatives and get an exponential convergence. Furthermore, we present a new type of viscosity terms depend on Riemann-Liouville fractional derivatives, and verified the effect of the terms in numerical tests. Our future work is to solve more FPDEs by building new spectral basis functions, which are complex in different kinds of derivative operators. 

\section*{Acknowledgments}
We appreciate the referees for their valuable suggestions. This work is partially supported by the National Natural Science Foundation of China (Grant No. U1637208), National Key Research and Development Program of China(2017YFB1401801), Natural Scientific Foundation of Heilongjiang Province in China (A2016003).
\begin{figure}[H]
\label{Fig5}
\centering
\subfigure{
\includegraphics[width=0.45\textwidth]{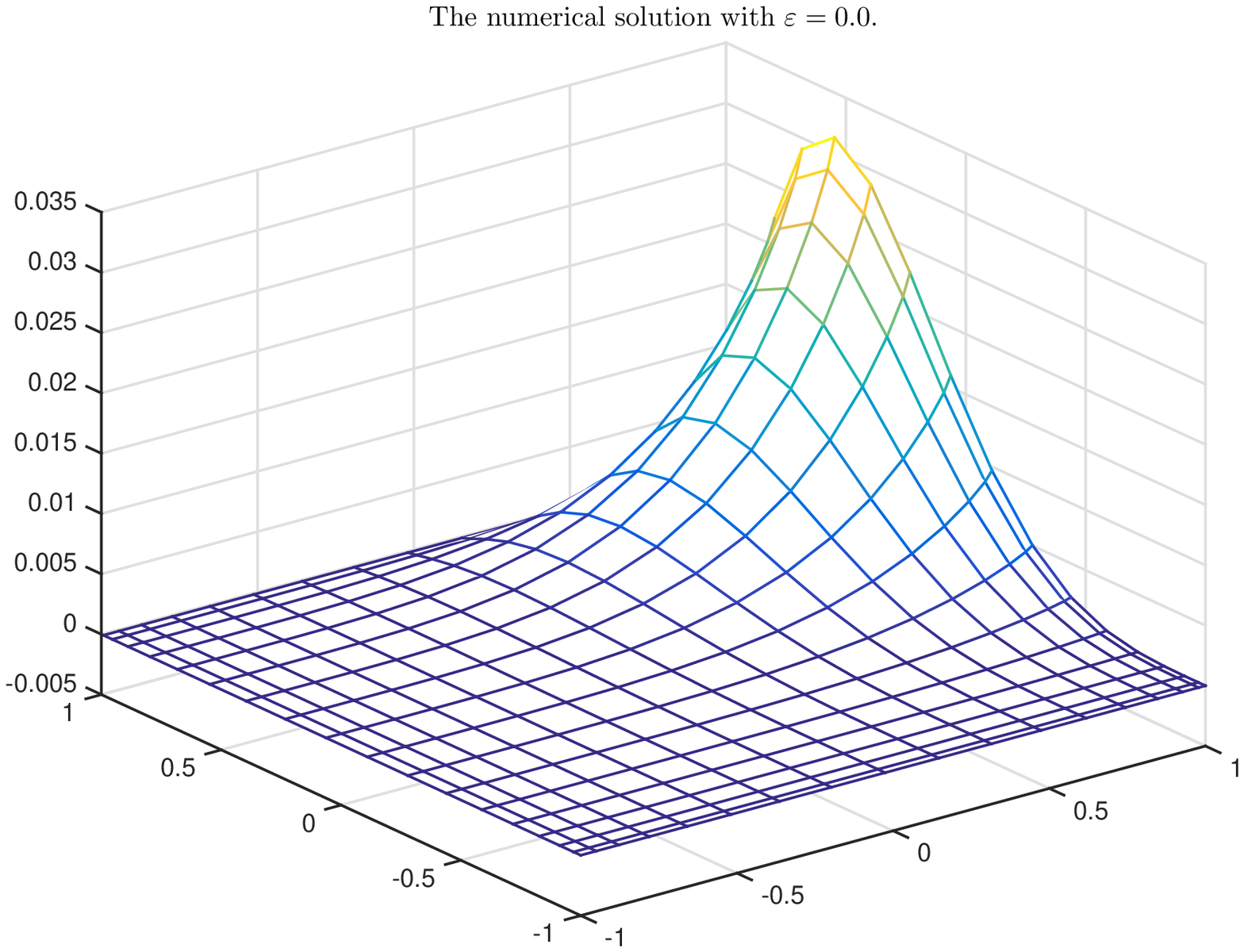}}
\subfigure{
\includegraphics[width=0.45\textwidth]{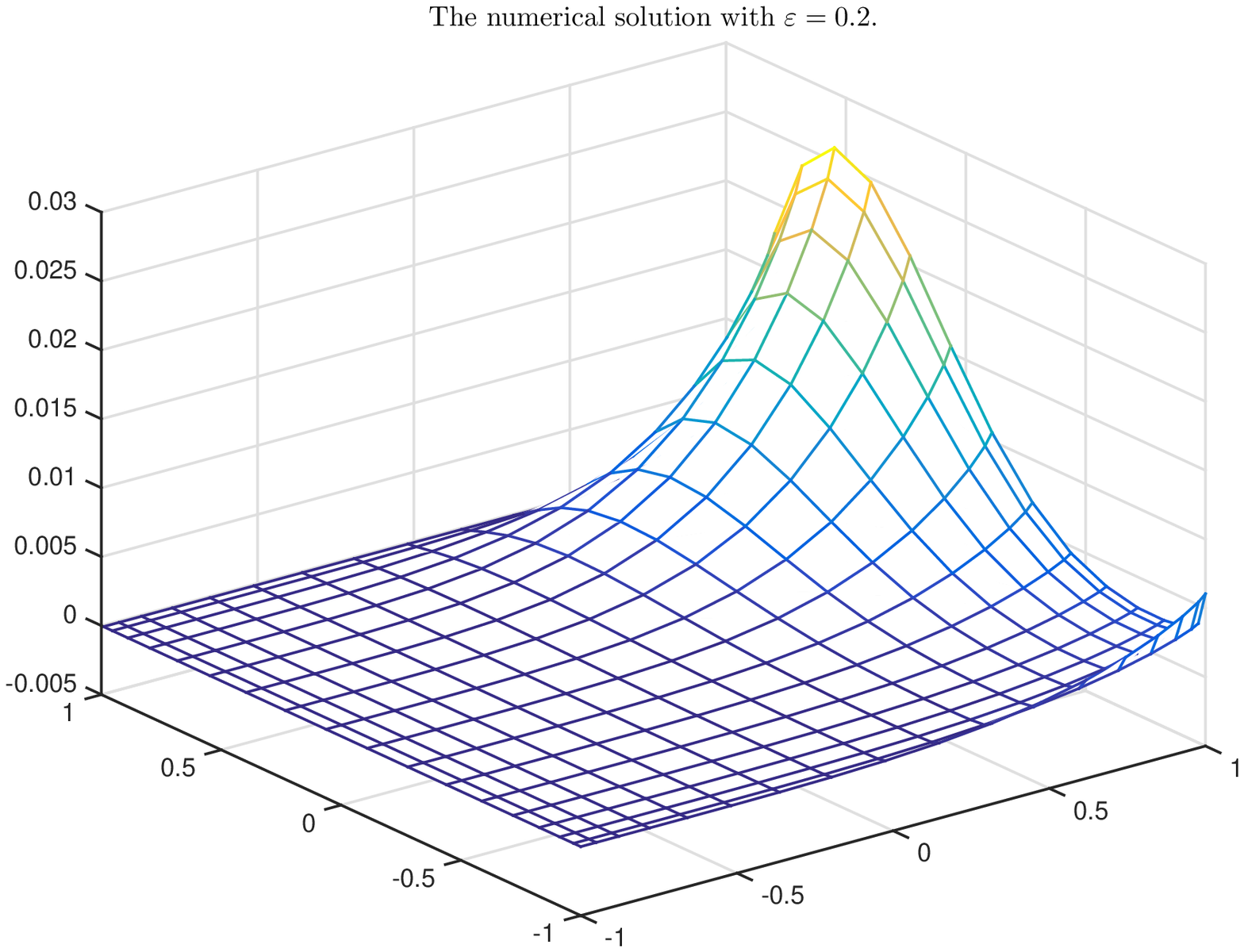}}
\subfigure{
\includegraphics[width=0.45\textwidth]{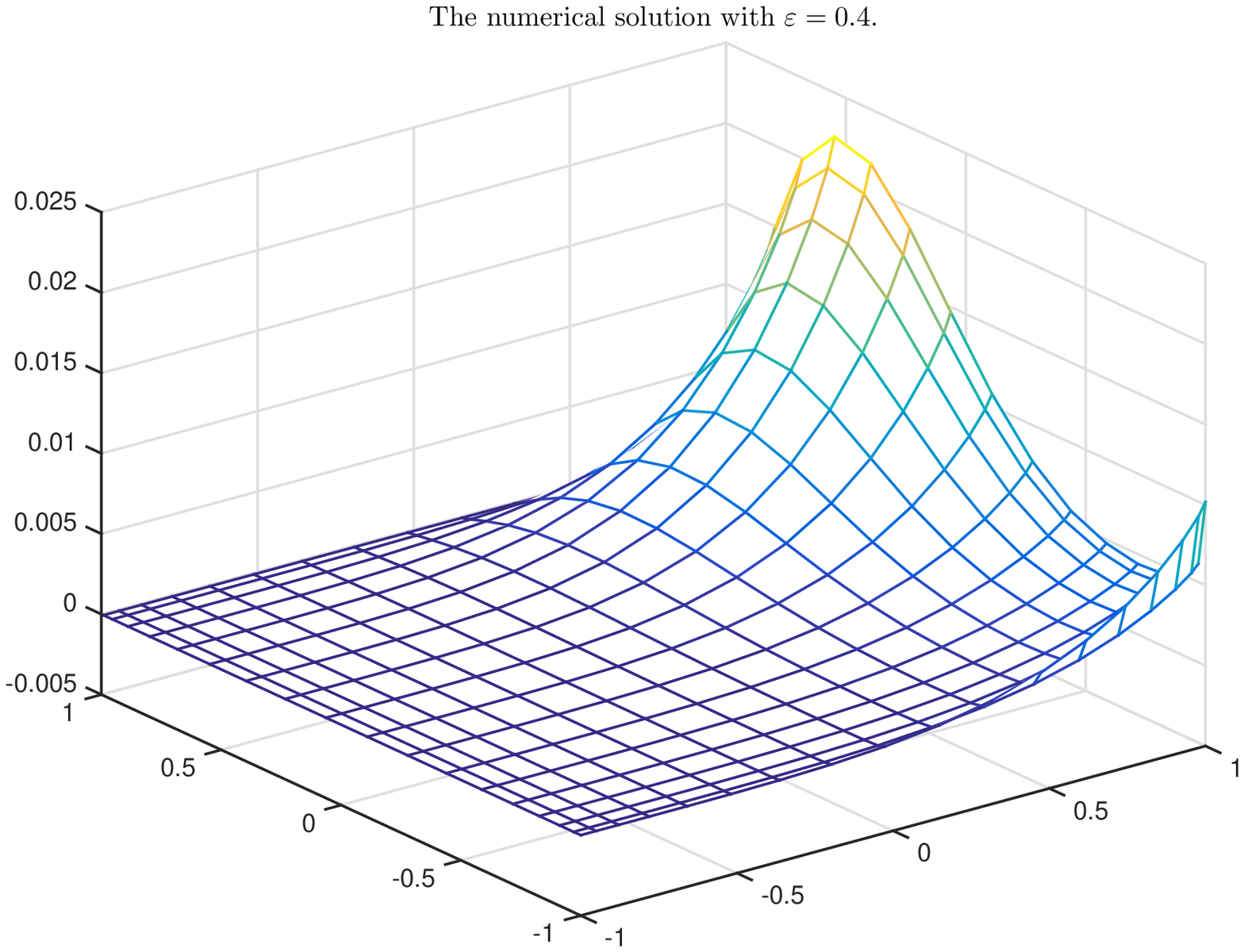}}
\subfigure{
\includegraphics[width=0.45\textwidth]{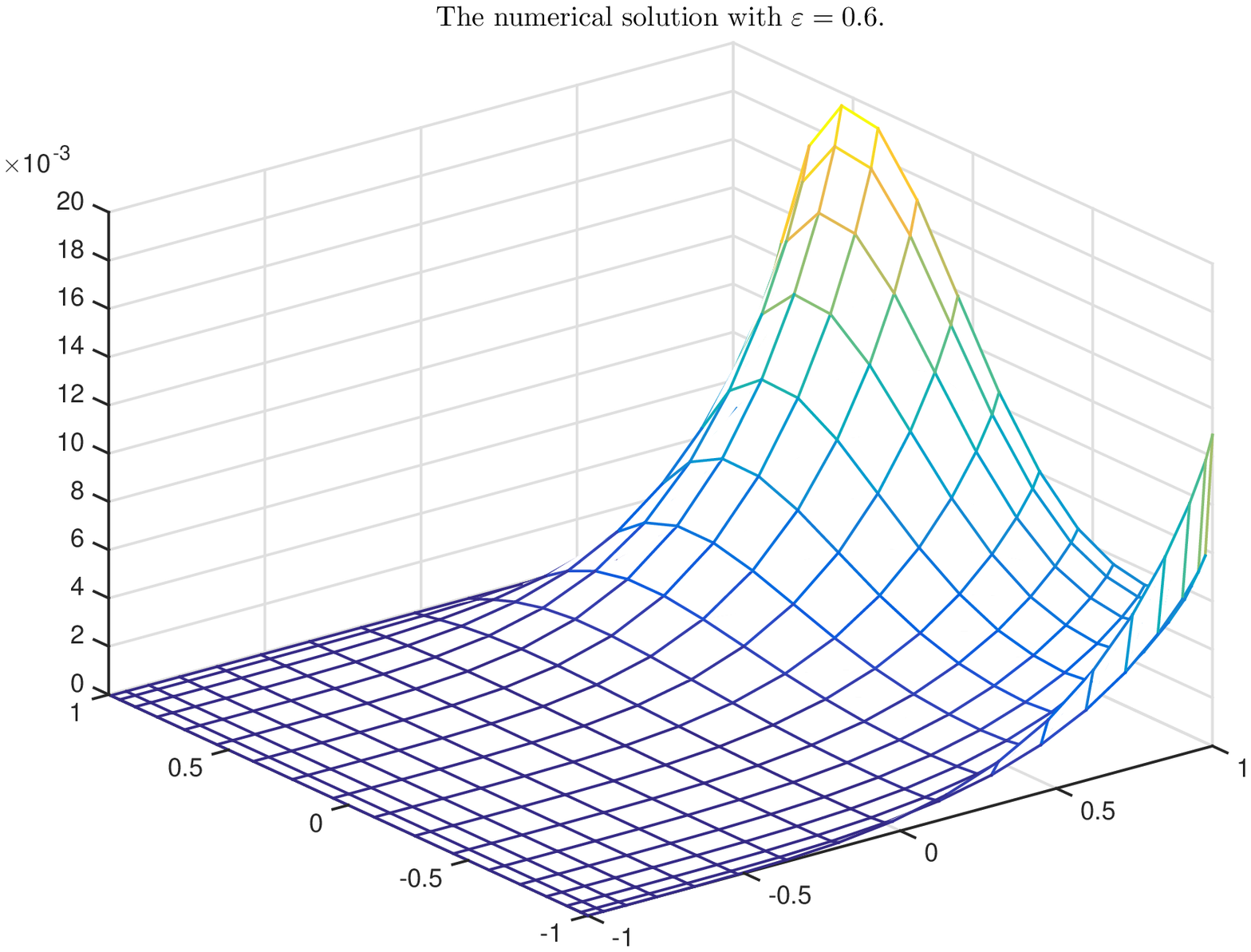}}
%\end{figure}
%\begin{figure}[H]
%\centering
\subfigure{
\includegraphics[width=0.45\textwidth]{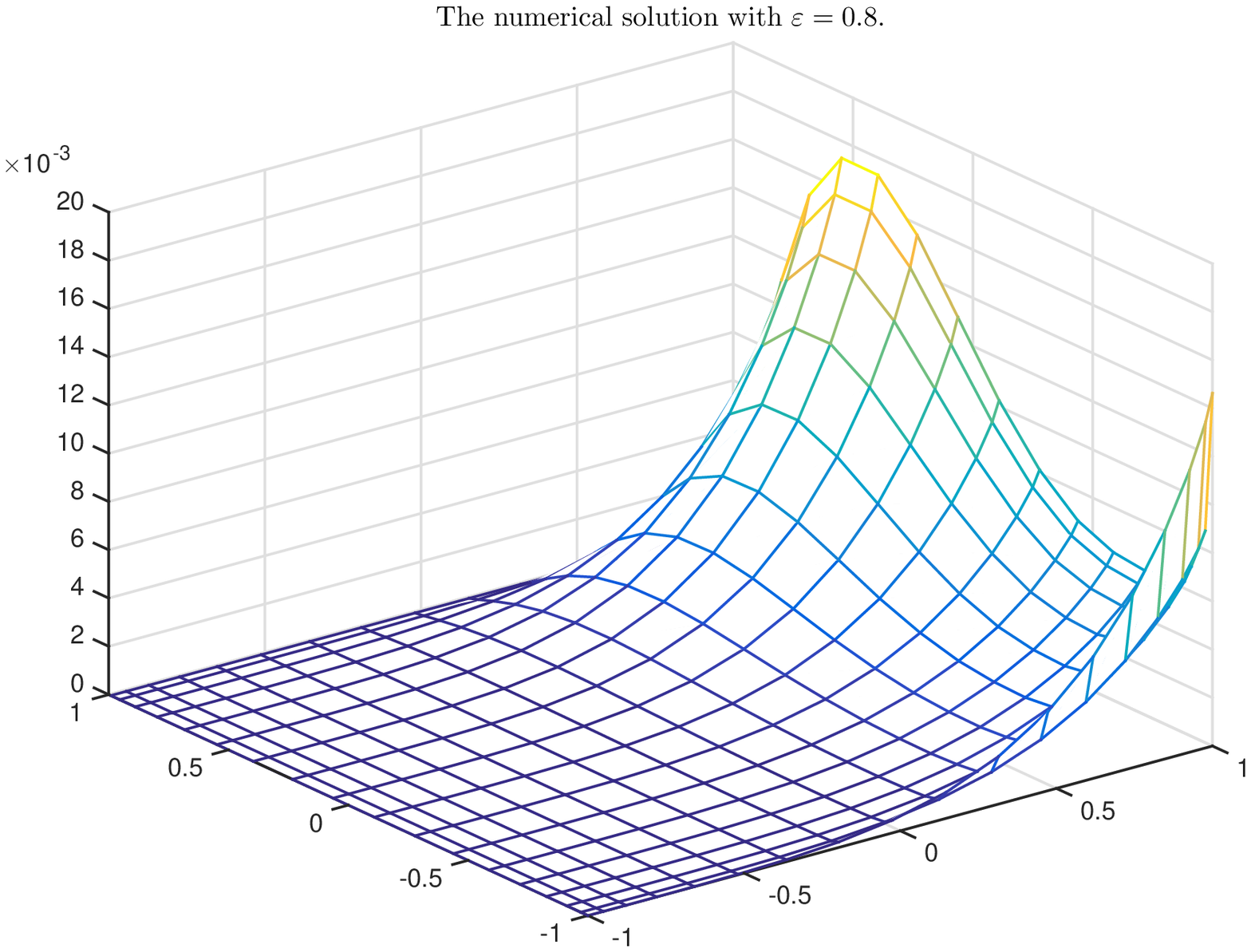}}
\subfigure{
\includegraphics[width=0.45\textwidth]{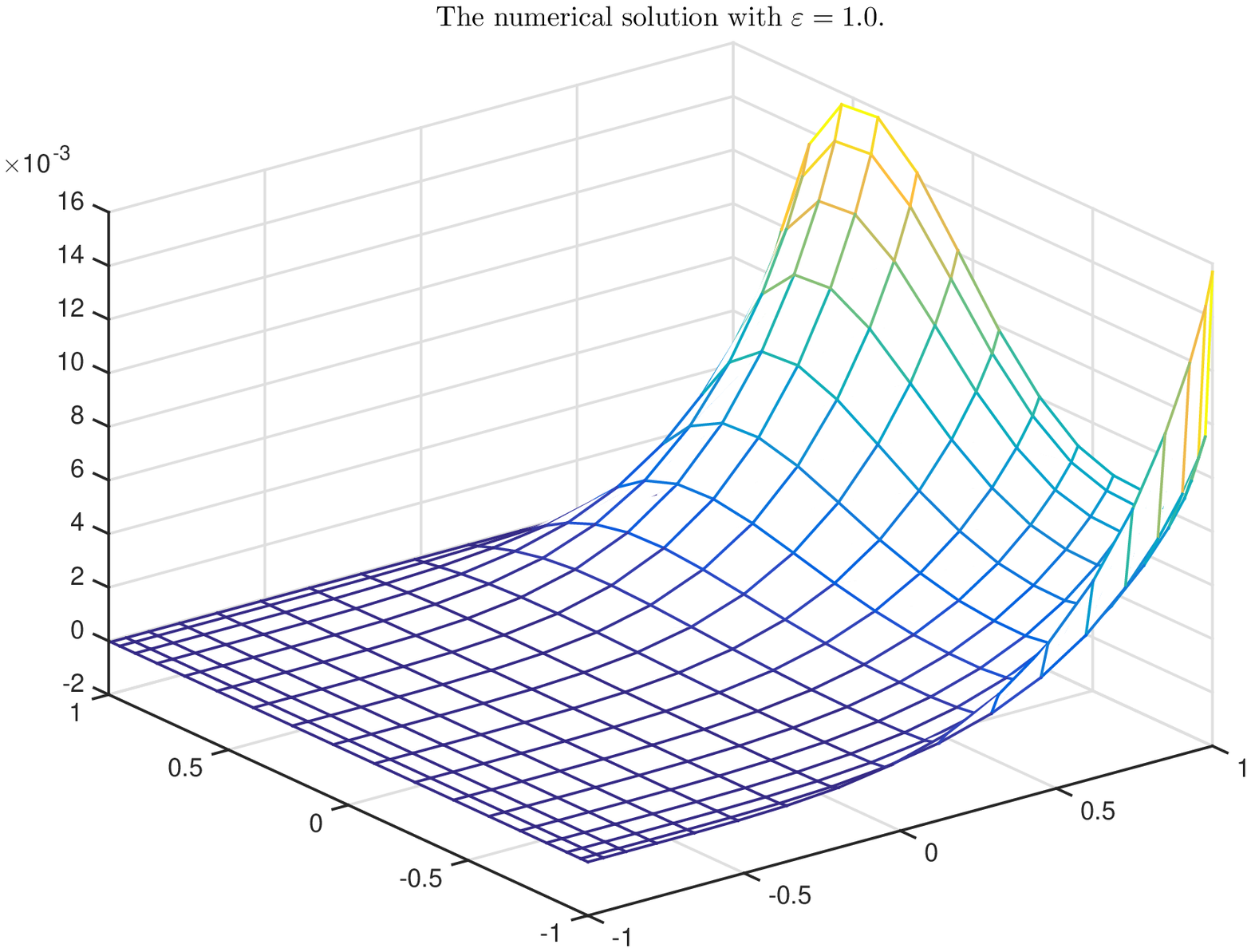}}
\caption{The numerical solutions with different parameters $\varepsilon$.}
\end{figure}

\begin{appendix}
\section{Proof of Theorem \ref{thm3-1} }
\label{AppendixA}

We first lead in two lemmas.

%\begin{lemmaA}
%\label{decreaselemma}
%Suppose $\alpha \in [0,1]$. The constant 
%\begin{align}
%\bar{\gamma}_0^{(\alpha,\alpha)} = \frac{2^{2\alpha+1}\Gamma^2(\alpha+1)}{(2\alpha+1)\Gamma(2\alpha+1)}
%\end{align}
%is strictly decreasing with variable $\alpha$.
%\end{lemmaA}
%\begin{proof}
%With the expression 6.1.18 in \cite{AS}, that is
%\begin{align*}
%\Gamma(2z) = (2\pi)^{-\frac{1}{2}}2^{2z-\frac{1}{2}}\Gamma(z)\Gamma\left(z+\frac{1}{2}\right),
%\end{align*}
%we rewrite the constant $\bar{\gamma}_0^{(\alpha,\alpha)}$ as follows:
%\begin{align*}
%\bar{\gamma}_0^{(\alpha,\alpha)} = \frac{2^{2\alpha+1}\Gamma^2(\alpha+1)}{\Gamma(2\alpha+2)} = \frac{\pi^{\frac{1}{2}}\Gamma(\alpha+1)}{\Gamma\left(\alpha+\frac{3}{2}\right)}.
%\end{align*}
%Denote
%\begin{align*}
%g(\alpha):=\ln\bar{\gamma}_0^{(\alpha,\alpha)} = \frac{1}{2}\ln\pi + \ln\Gamma(\alpha+1) - \ln\Gamma\left(\alpha+\frac{3}{2}\right),
%\end{align*}
%and recall the Psi (Digamma) function
%\begin{align*}
%\Psi(z) = \frac{{\rm d}}{{\rm d}z}\ln\Gamma(z) = \frac{\Gamma'(z)}{\Gamma(z)}, z \in \mathbb{R}^+,
%\end{align*}
%we have
%\begin{align*}
%\frac{{\rm d}}{{\rm d}\alpha}g(\alpha) = \Psi(\alpha+1) - \Psi\left(\alpha+\frac{3}{2}\right)<0,
%\end{align*}
%where the Psi function is increasing in $(0,2)$. So we get the conclusion.
%\end{proof}

\begin{lemmaA}
\label{increaselemma}
Denote $n \geqslant 1$ is a positive integer, $\alpha \in [0,1]$. The constant 
\begin{align}
\bar{\gamma}_n^{(\alpha,\alpha)} = \frac{2^{2\alpha+1}\Gamma^2(n+\alpha+1)}{(2n+2\alpha+1)n!\Gamma(n+2\alpha+1)}
\end{align}
is strictly increasing with variable $\alpha$.
\end{lemmaA}
\begin{proof}
Denote
\begin{align*}
g(\alpha) := & \ln \bar{\gamma}_n^{(\alpha,\alpha)} \\
= & (2\alpha+1)\ln2 + 2\ln\Gamma(n+\alpha+1) - \ln(2n+2\alpha+1) - \ln\Gamma(n+2\alpha+1) - \ln n!,
\end{align*}
which is smooth enough with variable $\alpha$. 
We have
\begin{align*}
\frac{{\rm d}}{{\rm d}\alpha}g(\alpha) = 2\ln2 + 2\Psi(n+\alpha+1) - \frac{2}{2n+2\alpha+1} - 2\Psi(n+2\alpha+1).
\end{align*}
Set $h(\alpha) = \frac{1}{2}\frac{{\rm d}}{{\rm d}\alpha}g(\alpha)$, we aim to prove the function $h(\alpha) $ is positive in $\alpha \in [0,1]$.

\textit{Case I: }$n \geqslant 2 $. With the expression 6.4.10 in \cite{AS},
\begin{align}
\label{diffpsi}
\frac{{\rm d}^n}{{\rm d}z^n} \Psi(z) = (-1)^{n+1}n!\sum_{k=0}^{+\infty}(z+k)^{-n-1}, z \neq 0,-1,-2,\cdots,
\end{align}
we have
\begin{align*}
h(\alpha) & = \ln2 - \frac{1}{2n+2\alpha+1} - \sum_{k=0}^{+\infty}\frac{1}{n+\alpha+1+k} + \sum_{k=0}^{+\infty}\frac{1}{n+2\alpha+1+k} \\
& = \ln2 - \frac{1}{2n+2\alpha+1} -\frac{1}{n+\alpha+1} - \sum_{k=0}^{+\infty}\frac{1}{n+\alpha+2+k} + \sum_{k=0}^{+\infty}\frac{1}{n+2\alpha+1+k} \\
& = \ln2 - \frac{1}{2n+2\alpha+1} -\frac{1}{n+\alpha+1} + \sum_{k=0}^{+\infty}\frac{1-\alpha}{(n+2\alpha+1+k)(n+\alpha+2+k)} \\
& \geqslant \ln2 - \frac{1}{5} -\frac{1}{3}>0.
\end{align*}

\textit{Case II: }$n = 1 $. By the expression 6.3.5, 6.3.8, and 6.3.22 in \cite{AS},
\begin{align}
& \Psi(z+1) = \Psi(z) + \frac{1}{z}, \\
& \Psi(2z) = \frac{1}{2}\Psi(z) + \frac{1}{2}\Psi(z+ \frac{1}{2}) + \ln2, \\
& \Psi(z) = \int_0^1 \frac{1-t^{z-1}}{1-t} {\rm d}t - \gamma^*,
\end{align}
where $\gamma^* = 0.577\ 215\ 665\cdots$ is a constant, it appears
\begin{align*}
h(\alpha) & = \ln2 - \frac{1}{3+2\alpha} + \Psi(2+\alpha) - \Psi(2+2\alpha) \\
& = \ln2 - \frac{1}{3+2\alpha} + \Psi(1+\alpha) + \frac{1}{1+\alpha} - \frac{1}{2}\Psi(1+\alpha) - \frac{1}{2}\Psi(\frac{3}{2}+\alpha) - \ln2 \\
& = \frac{1}{1+\alpha} - \frac{1}{3+2\alpha} + \frac{1}{2}\Psi(1+\alpha) - \frac{1}{2}\Psi(\frac{3}{2}+\alpha) \\
& = \frac{1}{1+\alpha} - \frac{1}{3+2\alpha} + \frac{1}{2}\int_0^1 \frac{1-t^{\alpha} - 1 + t^{\alpha + \frac{1}{2}}}{1-t} {\rm d}t \\
& = \frac{1}{1+\alpha} - \frac{1}{3+2\alpha} - \frac{1}{2}\int_0^1 \frac{t^{\alpha}}{1+t^{\frac{1}{2}}} {\rm d}t \\
& := \frac{1}{1+\alpha} - \frac{1}{3+2\alpha} - I.
\end{align*}
Taking $x = t^{\frac{1}{2}}$ and ${\rm d}t = 2x{\rm d}x$, we get the integral $I$
\begin{align*}
I  = \frac{1}{2}\int_0^1 \frac{t^{\alpha}}{1+t^{\frac{1}{2}}} {\rm d}t = \int_0^1 \frac{x^{2\alpha + 1}}{1+x} {\rm d}x < \int_0^1 x^{2\alpha + 1} {\rm d}x = \frac{1}{2\alpha + 2}.
\end{align*}
Then we get
\begin{align*}
h(\alpha) = \frac{1}{1+\alpha} - \frac{1}{3+2\alpha} - I > \frac{1}{1+\alpha} - \frac{1}{3+2\alpha} -\frac{1}{2\alpha + 2} = \frac{1}{2(\alpha+1)(2\alpha+3)} > 0.
\end{align*}
\end{proof}

\begin{lemmaA}
Denote $n \geqslant 1 $ is a positive integer, $0\leqslant \alpha \leqslant \beta \leqslant 1 $. The constant $\bar{\gamma}_n^{(\alpha, \alpha)} $ satisfies the following relationship
\begin{align*}
\bar{\gamma}_n^{(\alpha, \alpha)} \leqslant \bar{\gamma}_n^{(\beta, \beta)} \leqslant 2^{2(\beta - \alpha)}\bar{\gamma}_n^{(\alpha, \alpha)}.
\end{align*}
\end{lemmaA}

\begin{proof}
The first inequality is proved by previous lemma. Denote 
\begin{align*}
g(\alpha) = \ln\frac{(2n+2\alpha+1)\Gamma(n+2\alpha+1)}{\Gamma^2(n+\alpha+1)},
\end{align*}
which is a positive function, and smooth in $[0,1]$. Set
\begin{align*}
h(\alpha) = \frac{1}{2}\frac{{\rm d}}{{\rm d}\alpha} g(\alpha) = \frac{1}{2n+2\alpha+1} + \Psi(n+2\alpha+1) - \Psi(n+\alpha+1).
\end{align*}
Since $\Psi(\cdot) $ is strictly increasing in $(0, +\infty) $, so $h(\alpha) > 0$, and the function $g(\alpha)$ is monotonic increasing in $[0,1]$, we have
\begin{align*}
2^{2(\beta - \alpha)}\frac{\bar{\gamma}_n^{(\alpha, \alpha)}}{\bar{\gamma}_n^{(\beta, \beta)}} = \frac{2n+2\beta+1}{2n+2\alpha+1}\cdot\frac{\Gamma^2(n+\alpha+1)\Gamma(n+2\beta+1)}{\Gamma^2(n+\beta+1)\Gamma(n+2\alpha+1)} = \frac{g(\beta)}{g(\alpha)} \geqslant 1,
\end{align*}
which get the conclusion.

\end{proof}

We here give the following proposition.
\begin{prop}
\label{prop1}
The constant 
\begin{align*}
\bar{C}_{ij}^{(\alpha,\beta,\gamma,\mu )} & = {\bar{\gamma}_{i-1}^{(\sigma,\sigma)}\bar{\gamma}_{j-1}^{(s-\frac{\alpha}{2},s-\frac{\alpha}{2})} - \bar{\gamma}_{i-1}^{(\sigma-\frac{\beta}{2},\sigma-\frac{\beta}{2})}\bar{\gamma}_{j-1}^{(s,s)} - \varepsilon\cdot\bar{\gamma}_{i-1}^{(\sigma-\frac{\mu}{2},\sigma-\frac{\mu}{2})}\bar{\gamma}_{j-1}^{(s-\frac{\gamma}{2},s-\frac{\gamma}{2})}+\bar{\gamma}_{i-1}^{(\sigma,\sigma)}\bar{\gamma}_{j-1}^{(s,s)}}
\end{align*}
is positive while $\varepsilon \in [0, \min\{2^{\gamma - \alpha}, 1\}] $, and $i,j = 2,3,\cdots$.
\end{prop}
\begin{proof}
The proof could be divided by two parts.

\textit{Case I:} $\gamma < \alpha$, we have
\begin{align*}
\bar{C}_{ij}^{(\alpha,\beta,\gamma,\mu )} & = {\bar{\gamma}_{i-1}^{(\sigma,\sigma)}\bar{\gamma}_{j-1}^{(s-\frac{\alpha}{2},s-\frac{\alpha}{2})} - \bar{\gamma}_{i-1}^{(\sigma-\frac{\beta}{2},\sigma-\frac{\beta}{2})}\bar{\gamma}_{j-1}^{(s,s)} - \varepsilon\cdot\bar{\gamma}_{i-1}^{(\sigma-\frac{\mu}{2},\sigma-\frac{\mu}{2})}\bar{\gamma}_{j-1}^{(s-\frac{\gamma}{2},s-\frac{\gamma}{2})}+\bar{\gamma}_{i-1}^{(\sigma,\sigma)}\bar{\gamma}_{j-1}^{(s,s)}} \\
& > \bar{\gamma}_{i-1}^{(\sigma,\sigma)}[\bar{\gamma}_{j-1}^{(s-\frac{\alpha}{2},s-\frac{\alpha}{2})} - \bar{\gamma}_{j-1}^{(s,s)} - 2^{\gamma - \alpha} \bar{\gamma}_{j-1}^{(s-\frac{\gamma}{2},s-\frac{\gamma}{2})}+\bar{\gamma}_{j-1}^{(s,s)}] \geqslant 0.
\end{align*}

\textit{Case II:} $\gamma \geqslant \alpha$, we get
\begin{align*}
\bar{C}_{ij}^{(\alpha,\beta,\gamma,\mu )} & = {\bar{\gamma}_{i-1}^{(\sigma,\sigma)}\bar{\gamma}_{j-1}^{(s-\frac{\alpha}{2},s-\frac{\alpha}{2})} - \bar{\gamma}_{i-1}^{(\sigma-\frac{\beta}{2},\sigma-\frac{\beta}{2})}\bar{\gamma}_{j-1}^{(s,s)} - \varepsilon\cdot\bar{\gamma}_{i-1}^{(\sigma-\frac{\mu}{2},\sigma-\frac{\mu}{2})}\bar{\gamma}_{j-1}^{(s-\frac{\gamma}{2},s-\frac{\gamma}{2})}+\bar{\gamma}_{i-1}^{(\sigma,\sigma)}\bar{\gamma}_{j-1}^{(s,s)}} \\
& > \bar{\gamma}_{i-1}^{(\sigma,\sigma)}[\bar{\gamma}_{j-1}^{(s-\frac{\alpha}{2},s-\frac{\alpha}{2})} - \bar{\gamma}_{j-1}^{(s,s)} - \bar{\gamma}_{j-1}^{(s-\frac{\gamma}{2},s-\frac{\gamma}{2})}+\bar{\gamma}_{j-1}^{(s,s)}] \geqslant 0.
\end{align*}
\end{proof}

From now on, we start to prove the Theorem \ref{thm3-1}.
\begin{proof}
We can prove this theorem by the Babu\v{s}ka Lax-Milgram lemma. 

\textit{(i) Continuity.}
We start with the fact that the bilinear 
$ {\mathscr{A}}_\varepsilon ^{( \alpha ,\beta ,\gamma ,\mu ) }( u,v ) $ is continuous on ${_0B^{\frac{\alpha }{2},\frac{\beta }{2},\frac{\gamma }{2},\frac{\mu }{2}}}( Q_T )$ $ \times {{}^0B^{\frac{\alpha }{2},\frac{\beta }{2},\frac{\gamma }{2},\frac{\mu }{2}}}( Q_T ) $. Hereafter, we use H\"{o}lder's equality to achieve this part.
\begin{align*}
\begin{aligned}
\left|\mathscr{A}_\varepsilon ^{(\alpha ,\beta ,\gamma ,\mu )}( {u,v} )\right| & = \left|{( {_0\partial _t^{\frac{\alpha }{2}}u,{_t}\partial _T^{\frac{\alpha }{2}}v} )_{Q_T}} - {( {_{ - 1}\partial _x^{\frac{\beta }{2}}u,{_x}\partial _1^{\frac{\beta }{2}}v} )_{Q_T}} - \varepsilon  \cdot {( {_0\partial _t^{\frac{\gamma }{2}}{_{ - 1}}\partial _x^{\frac{\mu }{2}}u,{_t}\partial _T^{\frac{\gamma }{2}}{_x}\partial _1^{\frac{\mu }{2}}v} )_{Q_T}}+ {( {u,v} ) }_{Q_T}\right| \\
& \lesssim {\left\| u \right\|_{{H^{\frac{\alpha }{2}}}( {I;{L^2}( \Lambda  )} )}}{\left\| v \right\|_{{H^{\frac{\alpha }{2}}}( {I;{L^2}( \Lambda  )} )}} + {\left\| u \right\|_{{L^2}( {I;H_0^{\frac{\beta }{2}}( \Lambda  )} )}}{\left\| v \right\|_{{L^2}( {I;H_0^{\frac{\beta }{2}}( \Lambda  )} )}}\\
&\;\;\;\;\;\;\;\;\;\;\; + \varepsilon  \cdot {\left\| u \right\|_{{H^{\frac{\gamma }{2}}}( {I;H_0^{\frac{\mu }{2}}( \Lambda  )} )}}{\left\| v \right\|_{{H^{\frac{\gamma }{2}}}( {I;H_0^{\frac{\mu }{2}}( \Lambda  )} )}} + {\left\| u \right\|_{{L^{2}}( {I;{L^2}( \Lambda  )} )}}{\left\| v \right\|_{{L^{2}}( {I;{L^2}( \Lambda  )} )}}\\
& \lesssim {\left\| u \right\|_{{B^{\frac{\alpha }{2},\frac{\beta }{2},\frac{\gamma }{2},\frac{\mu }{2}}}( Q_T )}}{\left\| v \right\|_{{B^{\frac{\alpha }{2},\frac{\beta }{2},\frac{\gamma }{2},\frac{\mu }{2}}}( Q_T )}}.
\end{aligned}
\end{align*}

\textit{(ii) Inf-sup condition.}
For this part, we say the bilinear $ \mathscr{A}_\varepsilon ^{(\alpha ,\beta ,\gamma ,\mu ) }( u,v ) $ satisfies the inf-sup condition, that is, for any $0 \neq u \in {{_0B^{\frac{\alpha }{2},\frac{\beta }{2},\frac{\gamma }{2},\frac{\mu }{2}}}( Q_T )}$,
\begin{align}
\label{inf-sup}
\sup_{0 \neq v \in {{{}^0B^{\frac{\alpha }{2},\frac{\beta }{2},\frac{\gamma }{2},\frac{\mu }{2}}}( Q_T )}}
\frac{\left| \mathscr{A}_\varepsilon ^{(\alpha ,\beta ,\gamma ,\mu ) }( {u,v} )  \right|}{{\left\| u \right\|}
_{B^{\frac{\alpha }{2},\frac{\beta }{2},\frac{\gamma }{2},\frac{\mu }{2}}( Q_T )}{\left\| v \right\|}_{B^{\frac{\alpha }{2},\frac{\beta }{2},\frac{\gamma }{2},\frac{\mu }{2}}( Q_T )}} \geqslant \bar{\eta} > 0.
\end{align}

For this purpose, we construct $u \in {_0B^{\frac{\alpha }{2},\frac{\beta }{2},\frac{\gamma }{2},\frac{\mu }{2}}}( Q_T )  $ with the expansion
\begin{align*}
u(x,t) = \sum^{\infty}_{i=1}\sum^{\infty}_{j=1}u_{ij}\varphi_i(x)\psi_j(t).
\end{align*}
%where the condition \eqref{weakformulationcondition} shows that $u_{i1} = u_{1j} = 0, i,j = 1,2,\cdots $. 
Hence, for $0 \leqslant r\leqslant s, \ 0 \leqslant \rho \leqslant \sigma $, we get the expansion of solution $u$ with fractional-order operator as 
\begin{align*}
\begin{aligned}
_0\partial^r_t{}_{-1}\partial^\rho_xu(x,t) = &  \sum^{\infty}_{i=1}\sum^{\infty}_{j=1}u_{ij}\cdot{}_{-1}\partial^{\rho}_x\varphi_i(x)\;{}_0\partial^r_t\psi_j(t) \\
 = & \sum^{\infty}_{i=1}\sum^{\infty}_{j=1}u_{ij}\frac{\Gamma(i)}{\Gamma(i+\sigma-\rho)} {(1+x)}^{\sigma-\rho}\\
&\ \ \ \ \ \cdot \left[{P_{i-1}^{({-\sigma+\rho,\sigma-\rho})}(x)} - \frac{i({i+\sigma})}{(i+\sigma-\rho)(i-\sigma)}P_{i}^{({-\sigma+\rho,\sigma-\rho})}(x)\right] \\
&\ \ \ \ \ \cdot \frac{\Gamma(j)}{\Gamma(j+s-r)}{\left(j-\frac{1}{2}\right)}^{\frac{1}{2}}{\left(\frac{2}{T}\right)}^{s-r-\frac{1}{2}}t^{s-r}{P_{j-1}^{({-s+r,s-r})}\left(\frac{2t}{T}-1\right).}
\end{aligned}
\end{align*}
Suppose $u_{0j} = 0, j = 1,2,\cdots$, we have
\begin{align*}
_0\partial^r_t{}_{-1}\partial^\rho_xu(x,t) := \sum^{\infty}_{i=2}\sum^{\infty}_{j=2}{\hat{u}}_{ij}{(1+x)}^{\sigma-\rho}{P_{i-1}^{({-\sigma+\rho,\sigma-\rho})}(x)}{\left(\frac{2}{T}\right)}^{s-r-\frac{1}{2}}t^{s-r}{P_{j-1}^{({-s+r,s-r})}\left(\frac{2t}{T}-1\right),}
\end{align*}
where
\begin{align*}
{\hat{u}}_{ij} = \frac{\Gamma(i)}{\Gamma(i+\sigma-\rho)}\frac{\Gamma(j)}{\Gamma(j+s-r)}{\left(j-\frac{1}{2}\right)}^{\frac{1}{2}}\left({u_{ij} - \frac{i+\sigma}{i-\sigma}u_{i-1,j}}\right),\ \ \ i,j = 2,3,\cdots.
\end{align*}
Similarly, we could get the expansion of test function $v$ in the same way, that is,
\begin{align*}
\begin{aligned}
_t\partial^r_T{}_{x}\partial^\rho_{1}v(x,t) = &  \sum^{\infty}_{m=1}\sum^{\infty}_{n=1}v_{mn}\cdot{}_{x}\partial^{\rho}_1{\bar \varphi}_m(x)\;{}_t\partial^r_T{\bar \psi}_n(t) \\
 = & \sum^{\infty}_{m=1}\sum^{\infty}_{n=1}v_{mn}\frac{\Gamma(m)}{\Gamma(m+\sigma-\rho)} {(1-x)}^{\sigma-\rho}\\
&\ \ \ \ \ \cdot \left[{P_{m-1}^{({\sigma-\rho,-\sigma+\rho})}(x)} + \frac{m({m+\sigma})}{(m+\sigma-\rho)(m-\sigma)}P_{m}^{({\sigma-\rho,-\sigma+\rho})}(x)\right] \\
&\ \ \ \ \ \cdot \frac{\Gamma(n)}{\Gamma(n+s-r)}{\left(n-\frac{1}{2}\right)}^{\frac{1}{2}}{\left(\frac{2}{T}\right)}^{s-r-\frac{1}{2}}{(T-t)}^{s-r}{P_{n-1}^{({s-r,-s+r})}\left(\frac{2t}{T}-1\right)} \\
:= & \sum^{\infty}_{m=1}\sum^{\infty}_{n=1}{\hat{v}}_{mn}{(1-x)}^{\sigma-\rho}{P_{m-1}^{({\sigma-\rho,-\sigma+\rho})}(x)}{\left(\frac{2}{T}\right)}^{s-r-\frac{1}{2}}{\left(T-t\right)}^{s-r}{P_{n-1}^{({s-r,-s+r})}\left(\frac{2t}{T}-1\right)}.
\end{aligned}
\end{align*}
For the sake of convenience, denote $v_{0n} = 0, \ n=1,2,\cdots$, which means
\begin{align*}
{\hat{v}}_{mn} = \frac{\Gamma(m)}{\Gamma(m+\sigma-\rho)}\frac{\Gamma(n)}{\Gamma(n+s-r)}{\left(n-\frac{1}{2}\right)}^{\frac{1}{2}}\left({v_{mn} + \frac{m+\sigma}{m-\sigma}v_{m-1,n}}\right),\ \ \ m,n = 1,2,\cdots.
\end{align*}

To calculate the norm and bilinear form $\mathscr{A}_\varepsilon ^{(\alpha ,\beta ,\gamma ,\mu )}( {u,v} )$ expediently, it is better to transform the parameters of basis functions in Jacobi polynomials with different demands as follows:
\begin{align}
\label{uexpa1}
_0\partial^r_t{}_{-1}\partial^\rho_xu(x,t) = & \sum^{\infty}_{i=1}\sum^{\infty}_{j=1}{\hat{u}}_{ij}{(1+x)}^{\sigma-\rho}{P_{i-1}^{({-\sigma+\rho,\sigma-\rho})}(x)}{\left(\frac{2}{T}\right)}^{s-r-\frac{1}{2}}t^{s-r}{P_{j-1}^{({-s+r,s-r})}\left(\frac{2t}{T}-1\right)} \\
\label{uexpa2}
= & \sum^{\infty}_{i=1}\sum^{\infty}_{j=1}{\bar{u}}_{ij}{(1+x)}^{\sigma-\rho}{P_{i-1}^{({\sigma-\rho,\sigma-\rho})}(x)}{\left(\frac{2}{T}\right)}^{s-r-\frac{1}{2}}t^{s-r}{P_{j-1}^{({s-r,s-r})}\left(\frac{2t}{T}-1\right)} \\
\label{uexpa3}
= & \sum^{\infty}_{i=1}\sum^{\infty}_{j=1}{\tilde{u}}_{ij}{(1+x)}^{\sigma-\rho}{P_{i-1}^{({0,2(\sigma-\rho)})}(x)}{\left(\frac{2}{T}\right)}^{s-r-\frac{1}{2}}t^{s-r}{P_{j-1}^{({0,2(s-r)})}\left(\frac{2t}{T}-1\right)},\\
\label{vexpa1}
_t\partial^r_T{}_{x}\partial^\rho_{1}v(x,t) = & \sum^{\infty}_{m=1}\sum^{\infty}_{n=1}{\hat{v}}_{mn}{(1-x)}^{\sigma-\rho}{P_{m-1}^{({\sigma-\rho,-\sigma+\rho})}(x)}{\left(\frac{2}{T}\right)}^{s-r-\frac{1}{2}}{(T-t)}^{s-r}{P_{n-1}^{({s-r,-s+r})}\left(\frac{2t}{T}-1\right)} \\
\label{vexpa2}
= & \sum^{\infty}_{m=1}\sum^{\infty}_{n=1}{\bar{v}}_{mn}{(1-x)}^{\sigma-\rho}{P_{m-1}^{({\sigma-\rho,\sigma-\rho})}(x)}{\left(\frac{2}{T}\right)}^{s-r-\frac{1}{2}}{(T-t)}^{s-r}{P_{n-1}^{({s-r,s-r})}\left(\frac{2t}{T}-1\right)} \\
\label{vexpa3}
= & \sum^{\infty}_{m=1}\sum^{\infty}_{n=1}{\tilde{v}}_{mn}{(1-x)}^{\sigma-\rho}{P_{m-1}^{({2(\sigma-\rho),0})}(x)}{\left(\frac{2}{T}\right)}^{s-r-\frac{1}{2}}{(T-t)}^{s-r}{P_{n-1}^{({2(s-r),0})}\left(\frac{2t}{T}-1\right)}.
\end{align}
{\color{blue}
Denote the special test function 
\begin{align}
v^*(x,t) = & \sum^{M}_{m=2}\sum^{N}_{n=2}\frac{1}{\bar{C}_{mn}^{(\alpha,\beta,\gamma,\mu )}} {\bar{v}}_{mn}{(1-x)}^{\sigma}{P_{m-1}^{({\sigma,\sigma})}(x)}{\left(\frac{2}{T}\right)}^{s-\frac{1}{2}}{(T-t)}^{s}{P_{n-1}^{({s,s})}\left(\frac{2t}{T}-1\right)} \\
= & \sum^{M}_{m=1}\sum^{N}_{n=1}{\hat{v}}_{mn}{(1-x)}^{\sigma}{P_{m-1}^{({\sigma,-\sigma})}(x)}{\left(\frac{2}{T}\right)}^{s-\frac{1}{2}}{(T-t)}^{s}{P_{n-1}^{({s,-s})}\left(\frac{2t}{T}-1\right)}, \\
= & \sum^{M}_{m=1}\sum^{N}_{n=1}{\tilde{v}}_{mn}{(1-x)}^{\sigma}{P_{m-1}^{({2\sigma,0})}(x)}{\left(\frac{2}{T}\right)}^{s-\frac{1}{2}}{(T-t)}^{s}{P_{n-1}^{({2s,0})}\left(\frac{2t}{T}-1\right)},
\end{align}
where ${\bar{v}}_{mn} = {\bar{u}}_{mn}$, the constant $\bar{C}_{ij}^{(\alpha,\beta,\gamma,\mu )}$ (assume the constant is not zero, if so, the term will vanish in calculating the bilinear form) is as follows
\begin{align*}
\bar{C}_{ij}^{(\alpha,\beta,\gamma,\mu )} = \bar{\gamma}_{i-1}^{(\sigma,\sigma)}\bar{\gamma}_{j-1}^{(s-\frac{\alpha}{2},s-\frac{\alpha}{2})} - \bar{\gamma}_{i-1}^{(\sigma-\frac{\beta}{2},\sigma-\frac{\beta}{2})}\bar{\gamma}_{j-1}^{(s,s)} - \varepsilon\cdot\bar{\gamma}_{i-1}^{(\sigma-\frac{\mu}{2},\sigma-\frac{\mu}{2})}\bar{\gamma}_{j-1}^{(s-\frac{\gamma}{2},s-\frac{\gamma}{2})}+\bar{\gamma}_{i-1}^{(\sigma,\sigma)}\bar{\gamma}_{j-1}^{(s,s)}
\end{align*}
%$ {\hat{v}}_{mn} = {\hat{u}}_{mn}, {\tilde{v}}_{mn} = {\tilde{u}}_{mn}$, 
and the function $v^* \rightarrow 0$ under either of the following conditions: $t \rightarrow T, x \rightarrow -1$, or $x \rightarrow 1$.} 
Furthermore, we can get the scheme with multipled fractional derivatives
\begin{align}
_0\partial^r_t{}_{-1}\partial^\rho_xu_L(x,t) = & \sum^{M}_{i=1}\sum^{N}_{j=1}{\hat{u}}_{ij}{(1+x)}^{\sigma-\rho}{P_{i-1}^{({-\sigma+\rho,\sigma-\rho})}(x)}{\left(\frac{2}{T}\right)}^{s-r-\frac{1}{2}}t^{s-r}{P_{j-1}^{({-s+r,s-r})}\left(\frac{2t}{T}-1\right)}, \\
_t\partial^r_T{}_{x}\partial^\rho_{1}v^*(x,t) = & \sum^{M}_{m=1}\sum^{N}_{n=1}{\hat{v}}_{mn}{(1-x)}^{\sigma-\rho}{P_{m-1}^{({\sigma-\rho,-\sigma+\rho})}(x)}{\left(\frac{2}{T}\right)}^{s-r-\frac{1}{2}}{(T-t)}^{s-r}{P_{n-1}^{({s-r,-s+r})}\left(\frac{2t}{T}-1\right)},
\end{align}
$ (\bar{u}_{ij}, \ \tilde{u}_{ij}, \ \bar{u}_{mn}, \ \tilde{u}_{mn} \ resp.)$.
Thanks to the equivalence of norms with different weights in finite dimensions {\color{blue} (by $u_L$ and $v^*$)}, we can get the relationship of coefficients 
\begin{align}
\label{coefuequi}
\sum^{M}_{i=1}\sum^{N}_{j=1}{\bar{u}}^2_{ij} \cong \sum^{M}_{i=1} & \sum^{N}_{j=1}{\hat{u}}^2_{ij} \cong \sum^{M}_{i=1}\sum^{N}_{j=1}{\tilde{u}}^2_{ij}, \\
\label{coefvequi}
\sum^{M}_{m=2}\sum^{N}_{n=2}{\bar{u}}^2_{mn} \cong \sum^{M}_{m=1} & \sum^{N}_{n=1}{\hat{v}}^2_{mn} \cong \sum^{M}_{m=1}\sum^{N}_{n=1}{\tilde{v}}^2_{mn}.
\end{align}
Due to the fact that
\begin{align*}
\sum^{M}_{i=2}\sum^{N}_{j=2}{\bar{u}}^2_{ij} & \leqslant \sum^{M}_{i=1}\sum^{N}_{j=1}{\bar{u}}^2_{ij} \\
& \leqslant \left(1+ \left(\sum^{M}_{i=2}\sum^{N}_{j=2}{\bar{u}}^2_{ij}\right)^{-1}\left(\sum_{j=1}^N\bar{u}_{1j}^2+\sum_{i=2}^M\bar{u}_{i1}^2\right) \right)\sum^{M}_{i=2}\sum^{N}_{j=2}{\bar{u}}^2_{ij} \\
& \lesssim \sum^{M}_{i=2}\sum^{N}_{j=2}{\bar{u}}^2_{ij} ,
\end{align*}
we have
\begin{align*}
\sum^{M}_{i=1}\sum^{N}_{j=1}{\bar{u}}^2_{ij} \cong \sum^{M}_{m=2}\sum^{N}_{n=2}{\bar{u}}^2_{mn}.
\end{align*}
On the other hand, we could calculate the bilinear form {\color{blue} (by orthogonality of trial and test functions)}
\begin{align}
\label{innerproduct}
{({_0\partial^r_t{}_{-1}\partial^\rho_xu_L},{_t\partial^r_T{}_{x}\partial^\rho_{1}v^* })}_{Q_T} & =\sum^{M}_{i=2}\sum^{N}_{j=2}\frac{1}{\bar{C}_{ij}^{(\alpha,\beta,\gamma,\mu )}}\bar{u}_{ij}^2\bar{\gamma}_{i-1}^{(\sigma-\rho,\sigma-\rho)}\bar{\gamma}_{j-1}^{(s-r,s-r)}, 
\end{align}
and the $L^2$-norm of expansions of functions $u,v^*$ with fractional-order operator
\begin{align}
\label{unorm}
& {\left\| {_0\partial^r_t{}_{-1}\partial^\rho_xu_L}\right\|}_{L^2(Q_T)}^2 = \sum^{M}_{i=1}\sum^{N}_{j=1}{\tilde{u}}^2_{ij}{{\bar{\gamma}}_{i-1}^{(0,2(\sigma-\rho))}}{{\bar{\gamma}}_{j-1}^{(0,2(s-r))}}, \\
\label{vnorm}
& {\left\| {_t\partial^r_T{}_{x}\partial^\rho_{1}v^*}\right\|}_{L^2(Q_T)}^2 = \sum^{M}_{m=1}\sum^{N}_{n=1}{\tilde{v}}^2_{mn}{{\bar{\gamma}}_{m-1}^{(2(\sigma-\rho),0)}}{{\bar{\gamma}}_{n-1}^{(2(s-r),0)}}. 
\end{align}
By formula \eqref{vnorm},  it performs that
\begin{align*}
& {\left\| {_t\partial^{\frac{\alpha}{2}}_Tv^*}\right\|}_{L^2(Q_T)}^2 = \sum^{M}_{m=1}\sum^{N}_{n=1}{\tilde{v}}^2_{mn}{{\bar{\gamma}}_{m-1}^{(2\sigma,0)}}{{\bar{\gamma}}_{n-1}^{(2s-\alpha,0)}},  \\
& {\left\| {{}_x\partial^{\frac{\beta}{2}}_{1}v^*}\right\|}_{L^2(Q_T)}^2 = \sum^{M}_{m=1}\sum^{N}_{n=1}{\tilde{v}}^2_{mn}{{\bar{\gamma}}_{m-1}^{(2\sigma-\beta,0)}}{{\bar{\gamma}}_{n-1}^{(2s,0)}}, \\
& {\left\| {_t\partial^{\frac{\gamma}{2}}_T{}_{x}\partial^{\frac{\mu}{2}}_{1}v^*}\right\|}_{L^2(Q_T)}^2 = \sum^{M}_{m=1}\sum^{N}_{n=1}{\tilde{v}}^2_{mn}{{\bar{\gamma}}_{m-1}^{(2\sigma-\mu,0)}}{{\bar{\gamma}}_{n-1}^{(2s-\gamma,0)}}. 
\end{align*}
We see all of the above norms are positive and bounded, so we can verify that $v^* \in {}^0B^{\frac{\alpha}{2}, \frac{\beta}{2}, \frac{\gamma}{2}, \frac{\mu}{2}} (Q_T)$.

Based on formulas \eqref{innerproduct}-\eqref{vnorm}, we arrive at
\begin{align*}
{\left\| {u_L}\right\|}^2_{B^{\frac{\alpha }{2},\frac{\beta }{2},\frac{\gamma }{2},\frac{\mu }{2}}( Q_T ) } & = {\left\| {u_L}\right\|}^2_{H^{\frac{\alpha}{2}}(I;L^2(\Lambda))} + {\left\| {u_L}\right\|}^2_{L^{2}(I;H^{\frac{\beta}{2}}(\Lambda))} + {\left\| {u_L}\right\|}^2_{H^{\frac{\gamma}{2}}(I;H^{\frac{\mu}{2}}(\Lambda))} \\
& = \sum^{M}_{i=1}\sum^{N}_{j=1}{\tilde{u}}^2_{ij}\tilde{C}_{1,ij}^{(\alpha,\beta,\gamma,\mu )},\\
{\left\| {v^*}\right\|}^2_{B^{\frac{\alpha }{2},\frac{\beta }{2},\frac{\gamma }{2},\frac{\mu }{2}}( Q_T ) } & = {\left\| {v^*}\right\|}^2_{H^{\frac{\alpha}{2}}(I;L^2(\Lambda))} + {\left\| {v^*}\right\|}^2_{L^{2}(I;H^{\frac{\beta}{2}}(\Lambda))} + {\left\| {v^*}\right\|}^2_{H^{\frac{\gamma}{2}}(I;H^{\frac{\mu}{2}}(\Lambda))} \\
& = \sum^{M}_{i=1}\sum^{N}_{j=1}{\tilde{v}}^2_{ij}\tilde{C}_{2,ij}^{(\alpha,\beta,\gamma,\mu )},
\end{align*}
where
\begin{align*}
\tilde{C}_{1,ij}^{(\alpha,\beta,\gamma,\mu )} & = {{\bar{\gamma}}_{i-1}^{(0,2\sigma)}}{{\bar{\gamma}}_{j-1}^{(0,2s-\alpha)}} + {{\bar{\gamma}}_{i-1}^{(0,2\sigma-\beta)}}{{\bar{\gamma}}_{j-1}^{(0,2s)}} + {{\bar{\gamma}}_{i-1}^{(0,2\sigma-\mu)}}{{\bar{\gamma}}_{j-1}^{(0,2s-\gamma)}}, \\
\tilde{C}_{2,ij}^{(\alpha,\beta,\gamma,\mu )} & = {{\bar{\gamma}}_{i-1}^{(2\sigma,0)}}{{\bar{\gamma}}_{j-1}^{(2s-\alpha,0)}} + {{\bar{\gamma}}_{i-1}^{(2\sigma-\beta,0)}}{{\bar{\gamma}}_{j-1}^{(2s,0)}} + {{\bar{\gamma}}_{i-1}^{(2\sigma-\mu,0)}}{{\bar{\gamma}}_{j-1}^{(2s-\gamma,0)}},
\end{align*}
and the bilinear form
\begin{align*}
\mathscr{A}_\varepsilon ^{(\alpha ,\beta ,\gamma ,\mu ) }( {u_L,v^*} ) = & {( {_0\partial _t^{\frac{\alpha }{2}}u_L,{_t}\partial _T^{\frac{\alpha }{2}}v^*} )_{Q_T}} - {( {_{ - 1}\partial _x^{\frac{\beta }{2}}u_L,{_x}\partial _1^{\frac{\beta }{2}}v^*} )_{Q_T}} \\
& \ \ \ \ - \varepsilon  \cdot {( {_0\partial _t^{\frac{\gamma }{2}}{_{ - 1}}\partial _x^{\frac{\mu }{2}}u_L,{_t}\partial _T^{\frac{\gamma }{2}}{_x}\partial _1^{\frac{\mu }{2}}v^*} )_{Q_T}} + {( {u_L,v^*} ) }_{Q_T} \\
= & \sum^{M}_{i=2}\sum^{N}_{j=2}{\bar{u}}_{ij}^2.
\end{align*}
Denote
\begin{align*}
\tilde{C} = \max_{1 \leqslant i \leqslant M, 1 \leqslant j \leqslant N}\left\lbrace \tilde{C}_{1,ij}^{(\alpha,\beta,\gamma,\mu )} \right\rbrace,\ \hat{C} = \max_{1 \leqslant i \leqslant M, 1 \leqslant j \leqslant N}\left\lbrace \tilde{C}_{2,ij}^{(\alpha,\beta,\gamma,\mu )} \right\rbrace,
\end{align*}
we have
\begin{align*}
\mathscr{A}_\varepsilon ^{(\alpha ,\beta ,\gamma ,\mu ) }( {u_L,v^*} ) & \geqslant \sum^{M}_{i=2}\sum^{N}_{j=2}{\bar{u}}_{ij}^2 =  \left(\sum^{M}_{i=2}\sum^{N}_{j=2}{\bar{u}}_{ij}^2\right)^{\frac{1}{2}}\left(\sum^{M}_{i=2}\sum^{N}_{j=2}{\bar{u}}_{ij}^2\right)^{\frac{1}{2}} \\
& \geqslant \frac{1}{\tilde{C}\hat{C}} \left(\sum^{M}_{i=2}\sum^{N}_{j=2}{\tilde{C}}{\bar{u}}_{ij}^2\right)^{\frac{1}{2}} \left(\sum^{M}_{i=2}\sum^{N}_{j=2}{\hat{C}}{\bar{u}}_{ij}^2\right)^{\frac{1}{2}}  \\
& \cong \frac{1}{\tilde{C}\hat{C}} \left(\sum^{M}_{i=1}\sum^{N}_{j=1}{\tilde{C}}{\tilde{u}}_{ij}^2\right)^{\frac{1}{2}} \left(\sum^{M}_{i=1}\sum^{N}_{j=1}{\hat{C}}{\tilde{v}}_{ij}^2\right)^{\frac{1}{2}} \\
&  \geqslant \frac{1}{\tilde{C}\hat{C}}{\left\| {u_L}\right\|}_{B^{\frac{\alpha }{2},\frac{\beta }{2},\frac{\gamma }{2},\frac{\mu }{2}}( Q_T ) } {\left\| {v^*}\right\|}_{B^{\frac{\alpha }{2},\frac{\beta }{2},\frac{\gamma }{2},\frac{\mu }{2}}( Q_T ) }. 
\end{align*}
Letting $L \rightarrow (+\infty, +\infty) $, with the continuous of the bilinear form $\mathscr{A}_\varepsilon ^{(\alpha ,\beta ,\gamma ,\mu ) }( {\cdot,v^*} ) $, we get the inf-sup condition.

\textit{(iii) ``Transposed" inf-sup condition.}
To tackle this part, we construct formulas \eqref{vexpa1}-\eqref{vexpa3} for all $0 \neq v \in {}^0B^{\frac{\alpha }{2},\frac{\beta }{2},\frac{\gamma }{2},\frac{\mu }{2}} ( Q_T )$. Choose $0 \neq u^* \in {}_0B^{\frac{\alpha }{2},\frac{\beta }{2},\frac{\gamma }{2},\frac{\mu }{2}} ( Q_T )$ with $\bar{u}^*_{ij} = \bar{v}_{ij}, \tilde{u}^*_{ij}= \tilde{v}_{ij}$, we could get the ``transposed" inf-sup condition via the similar way with part \textit{(ii)}.
Thus, the well-posedness of problem \eqref{3-1} is proved.

Finally, owing to Cauchy-Schwarz inequality, we get the global estimate of the solution by part \textit{(ii)}
\[\left\| u \right\|_{{B^{\frac{\alpha }{2},\frac{\beta }{2},\frac{\gamma }{2},\frac{\mu }{2}}}( Q_T )}\left\| v^* \right\|_{{B^{\frac{\alpha }{2},\frac{\beta }{2},\frac{\gamma }{2},\frac{\mu }{2}}}( Q_T )}  \lesssim \mathscr{A}_\varepsilon ^{(\alpha ,\beta ,\gamma ,\mu )}( {u,v^*} ) = \left\langle {f,v^*} \right\rangle  \lesssim {\left\| f \right\|_{{{\left( {{B^{\frac{\alpha }{2},\frac{\beta }{2},\frac{\gamma }{2},\frac{\mu }{2}}}( Q_T )} \right)}^{\prime} }}}{\left\| v^* \right\|_{{B^{\frac{\alpha }{2},\frac{\beta }{2},\frac{\gamma }{2},\frac{\mu }{2}}}( Q_T )}},\]
which implies \eqref{3-2}.
\end{proof}
\end{appendix}


\section*{References}
\begin{thebibliography}{00}

\bibitem[\protect\citeauthoryear{Podlubny}{1998}]{IP}
I. Podlubny. Fractional Differential Equations: an Introduction to Fractional Derivatives, Fractional Differential Equations, to Methods of Their Solution and Some of Their Applications. Academic Press, 1998.


\bibitem[\protect\citeauthoryear{Barkai}{2000}]{BMK}
E. Barkai, R. Metzler, J. Klafter, From Continuous Time Random Walks to the Fractional Fokker-Planck Equation. Physical Review E, 2000, 61(1): 132--138.

\bibitem[\protect\citeauthoryear{Naber}{2004}]{NM}
M. Naber. Time Fractional Schr\"{o}dinger Equation. Journal of Mathematical Physics, 2004, 45(8): 3339--3352.

\bibitem[\protect\citeauthoryear{Tarasov}{2005}]{TZ}
V.E. Tarasov, G.M. Zaslavsky. Fractional Ginzburg-Landau Equation for Fractal Media. Physica A: Statistical Mechanics and its Applications, 2005, 354: 249--261.

\bibitem[\protect\citeauthoryear{Pu}{2010}]{PG}
X. Pu, B. Guo. Existence and Decay of Solutions to the Two-Dimensional Fractional Quasigeostrophic Equation. Journal of Mathematical Physics, 2010, 51(8): 771--45.

\bibitem[\protect\citeauthoryear{Guo}{2010}]{GZ}
B. Guo, M. Zeng. Solutions for the Fractional Landau-Lifshitz Equation. Journal of Mathematical Analysis and Applications, 2010, 361(1): 131--138.

\bibitem[\protect\citeauthoryear{Pu}{2012}]{PGZ}
X. Pu, B. Guo, J. Zhang. Global Weak Solutions to the 1-D Fractional Landau-Lifshitz equation. Discrete and Continuous Dynamical Systems-Series B (DCDS-B), 2012, 14(1):199--207.

\bibitem[\protect\citeauthoryear{Li}{2010}]{LX2}
X. Li, C. Xu. Existence and Uniqueness of the Weak Solution of the Space-Time Fractional Diffusion Equation and a Spectral Method Approximation. Communications in Computational Physics, 2010, 8(5):1016--1051.



\bibitem[\protect\citeauthoryear{Foy}{2010}]{LRF1}
L.R. Foy. Steady State Solutions of Hyperbolic Systems of Conservation Laws with Viscosity Terms. Communications on Pure \& Applied Mathematics, 2010, 17(2):177--188.

\bibitem[\protect\citeauthoryear{Foy}{1964}]{LRF2}
L.R. Foy. Steady State Solutions of Conservation Laws with Viscosity Terms. Communications on Pure \& Applied Mathematics, 1964, 17(2).

\bibitem[\protect\citeauthoryear{Lamballais}{2011}]{LFL}
E. Lamballais, V. Fortun\'{e}, S.Laizet. Straightforward High-Order Numerical Dissipation via the Viscous Term for Direct and Large Eddy Simulation. Journal of Computational Physics, 2011, 230(9):3270--3275.

\bibitem[\protect\citeauthoryear{Baldyga}{1987}]{BR}
J. Baldyga, S. Rohani. Micromixing Described in Terms of Inertial-Convective Disintegration of Large Eddies and Viscous-Convective Interactions among Small Eddies-I. General Development and Batch Systems. Chemical Engineering Science, 1987, 42(11):2597--2610.

\bibitem[\protect\citeauthoryear{Rohani}{1987}]{RB}
S. Rohani, J. Baldyga. Micromixing Described in Terms of Inertial-Convective Disintegration of Large Eddies and Viscous-Convective Interactions among Small Eddies-II. Semi-Batch and Continuous Stirred Tank Reactors. Chemical Engineering Science, 1987, 42(11):2611--2619.

\bibitem[\protect\citeauthoryear{Meerschaert}{1970}]{ST}
R.E. Showalter, T.W. Ting. Pseudoparabolic Partial Differential Equations. SIAM Journal on Mathematical Analysis, 1970, 88(1):1--26.

\bibitem[\protect\citeauthoryear{Mao}{2015}]{MS}
Z.P. Mao, J. Shen. A Semi-Implicit Spectral Deferred Correction Method for Two Water Wave Models with Nonlocal Viscous Term and Numerical Study of Their Decay Rates. Scientia Sinica, 2015, 45(8):1153--1168.

\bibitem[\protect\citeauthoryear{Zhang}{2012}]{ZSZ}
Y.N. Zhang, Z.Z. Sun, X. Zhao. Compact Alternating Direction Implicit Scheme for the Two-Dimensional Fractional Diffusion-Wave Equation. SIAM Journal on Numerical Analysis, 2012, 45(50):1535--1555.

\bibitem[\protect\citeauthoryear{Gao}{2015}]{GS}
G.H. Gao, Z.Z. Sun. A Compact Finite Difference Scheme for the Fractional Sub-Diffusion Equations. Journal of Scientific Computing, 2015, 64(3):959--985.

\bibitem[\protect\citeauthoryear{Wang}{2014}]{WYZ}
H. Wang, D. Yang, S. Zhu. Inhomogeneous Dirichlet Boundary-Value Problems of Space-Fractional Diffusion Equations and Their Finite Element Approximations. SIAM Journal on Numerical Analysis, 2014, 52(3):1292--1310.

\bibitem[\protect\citeauthoryear{Wang}{2017}]{WCW}
H. Wang, A. Cheng, K. Wang. Fast Finite Volume Methods for Space-Fractional Diffusion Equations. Discrete and Continuous Dynamical Systems-Series B (DCDS-B), 2017, 20(5):1427--1441.

\bibitem[\protect\citeauthoryear{Jia}{2015}]{JW}
J. Jia, H. Wang. Fast Finite Difference Methods for Space-Fractional Diffusion Equations with Fractional Derivative Boundary Conditions. Journal of Computational Physics, 2015, 293: 359--369.

\bibitem[\protect\citeauthoryear{Mustapha}{2014}]{MAF}
K. Mustapha, B. Abdallah, K.M. Furati. A Discontinuous Petrov-Galerkin Method for Time-Fractional Diffusion Equations. SIAM Journal on Numerical Analysis, 2014, 52(5): 2512--2529.

\bibitem[\protect\citeauthoryear{Shen}{2011}]{STW}
J. Shen, T. Tang, L.L. Wang. Spectral Methods: Algorithms, Analysis and Applications. Springer Science Business Media, 2011.

\bibitem[\protect\citeauthoryear{Shen}{2013}]{GSZ}
B.Y. Guo, T. Sun, C. Zhang. Jacobi and Laguerre Quasi-Orthogonal Approximations and Related Interpolations. Mathematics of Computation, 2013, 82(281):413--441.

\bibitem[\protect\citeauthoryear{Yu}{2014}]{YG}
X.H. Yu, B.Y. Guo. Spectral Element Method for Mixed Inhomogeneous Boundary Value Problems of Fourth Order. Journal of Scientific Computing, 2014, 61(3):673--701.

\bibitem[\protect\citeauthoryear{Wang}{2011}]{WZZ}
L.L. Wang, X.D. Zhao, Z.M. Zhang. Superconvergence of Jacobi-Gauss-Type Spectral Interpolation. Journal of Scientific Computing, 2014, 59(3):667--687.

\bibitem[\protect\citeauthoryear{Sheng}{2014}]{SWG}
C.T. Sheng, Z.Q. Wang, B.Y. Guo. A Multistep Legendre-Gauss Spectral Collocation Method for Nonlinear Volterra Integral Equations. SIAM Journal on Numerical Analysis, 2014, 52(4):1953--1980.

\bibitem[\protect\citeauthoryear{Li}{2009}]{LX}
X. Li, C. Xu. A Space-Time Spectral Method for the Time Fractional Diffusion Equation. SIAM Journal on Numerical Analysis, 2009, 47(3): 2108--2131.

\bibitem[\protect\citeauthoryear{Zhang}{2015}]{ZZ}
X. Zhao, Z. Zhang. Superconvergence Points of Fractional Spectral Interpolation. SIAM Journal on Scientific Computing, 2015, 38(1):A598--A613.

\bibitem[\protect\citeauthoryear{Jiao}{2015}]{JWH}
Y. Jiao, L.L. Wang, C. Huang. Well-Conditioned Fractional Collocation Methods Using Fractional Birkhoff Interpolation Basis. Journal of Computational Physics, 2016, 305(4):1--28.

\bibitem[\protect\citeauthoryear{Chen}{2015}]{CSW}
S. Chen, J. Shen, L.L. Wang. Generalized Jacobi Functions and Their Applications to Fractional Differential Equations. Mathematics of Computation, 2016, 85(300):1603--1638.

\bibitem[\protect\citeauthoryear{Chen}{2017}]{CSW2}
S. Chen, J. Shen, L.L. Wang. Laguerre Functions and Their Applications to Tempered Fractional Differential Equations on Infinite Intervals. Journal of Scientific Computing, 2017(6):1--28.

\bibitem[\protect\citeauthoryear{Shen}{2016}]{MS2}
Z.P. Mao, J. Shen. Efficient Spectral-Galerkin Methods for Fractional Partial Differential Equations with Variable Coefficients. Journal of Computational Physics, 2016(307):243--261.

\bibitem[\protect\citeauthoryear{Mao}{2017}]{MS3}
Z.P. Mao, J. Shen. Spectral Element Method with Geometric Mesh for Two-Sided Fractional Differential Equations. Advances in Computational Mathematics, 2017(3):1--27.

\bibitem[\protect\citeauthoryear{Mao}{2017}]{MS4}
Z.P. Mao, J. Shen. Hermite Spectral Methods for Fractional PDEs in Unbounded Domains. SIAM Journal on Scientific Computing, 2017, 39(5):A1928--A1950.

% \bibitem[\protect\citeauthoryear{Ezz}{2017}]{EHB}
% S.S. Ezz-Eldien, R.M. Hafez, A.H. Bhrawy, et al. New Numerical Approach for Fractional Variational Problems Using Shifted Legendre Orthonormal Polynomials. Journal of Optimization Theory \& Applications, 2017, 174:1--26.

% \bibitem[\protect\citeauthoryear{Hafez}{2017}]{HDB}
% R.M. Hafez, E.H. Doha, A.H. Bhrawy, et al. Numerical Solutions of Two-Dimensional Mixed Volterra-Fredholm Integral Equations via Bernoulli Collocation Method. Romanian Journal of Physics, 2017, 62(3).

% \bibitem[\protect\citeauthoryear{Bhrawy}{2017}]{BED}
% A.H. Bhrawy, S.S. Ezz-Eldien, E.H. Doha, et al. Solving Fractional Optimal Control Problems within a Chebyshev-Legendre Operational Technique. International Journal of Control, 2017, 90(6):1230--1244.

\bibitem[\protect\citeauthoryear{Ervin}{2006}]{EVR}
V.J. Ervin, J.P. Roop. Variational Formulation for the Stationary Fractional Advection-Dispersion Equation. Numerical Methods for Partial Differential Equations, 2006, 22(3): 558--576.

\bibitem[\protect\citeauthoryear{Andrews}{1999}]{AAR}
G.E. Andrews, R. Askey, R. Roy. Special Functions, Volume 71 of Encyclopedia of Mathematics and Its Applications. Cambridge University Press, Cambridge, 1999.

\bibitem[\protect\citeauthoryear{Andrews}{1965}]{AS}
M. Abramowitz, I.A. Stegun. Handbook of Mathematical Functions: with Formulas, Graphs, and Mathematical Tables. Courier Corporation, 1965.


\end{thebibliography}
\end{document}